\documentclass[11pt]{article}
\usepackage[left=3.2cm,right=3.2cm,top=3.5cm,bottom=3cm]{geometry}
\usepackage[utf8]{inputenc}
\usepackage[english]{babel}
\usepackage{amsmath}
\usepackage{amssymb}
\usepackage{amsthm}
\usepackage{newpxtext,newpxmath}
\usepackage{braket}
\usepackage{amscd}
\usepackage{setspace}
\usepackage{booktabs}
\usepackage{lscape}
\usepackage{rotating}
\usepackage{graphicx}
\usepackage{tikz}
\usepackage{xcolor}
\usepackage{mathtools}
\usepackage{sidecap}
\usepackage{emptypage}
\usepackage{enumitem}
\usepackage{todonotes}

\usepackage[colorlinks=true, allcolors=blue, pagebackref=true]{hyperref}

\newtheorem{theorem}{Theorem}
\newtheorem*{theorem*}{Theorem}

\newtheorem{lemma}{Lemma}
\newtheorem{proposition}{Proposition}
\newtheorem{corollary}{Corollary}
\newtheorem*{corollary*}{Corollary}

\newtheorem*{result*}{Result}
\theoremstyle{remark}
\newtheorem{remark}{Remark}
\theoremstyle{definition}

\newtheorem*{example*}{Example}

\newtheorem*{Ack}{Acknowledments}

\newcommand{\Z}{\mathbb{Z}}
\newcommand{\C}{\mathbb{C}}
\newcommand{\Q}{\mathbb{Q}}
\newcommand{\R}{\mathbb{R}}
\newcommand{\F}{\mathbb{F}}
\newcommand{\pro}{\mathbb{P}}
\newcommand{\Po}{\mathbb{H}}
\newcommand{\OO}{\mathcal{O}}
\newcommand{\fa}{\mathfrak{a}}

\newcommand{\fp}{\mathfrak{p}}

\newcommand{\Spec}{\mathrm{Spec}}

\newcommand{\SL}{\mathrm{SL}}

\newcommand{\spec}{\mathrm{Spec}}

\def\={\;=\;}
\def\:={\;:=\;}
\def\+{\;+\;}
\def\-{\;-\;}

\title{Partial Hasse invariants for genus zero curves in Hilbert modular varieties}
\author{Gabriele Bogo \and Yingkun Li}
\date{}
\begin{document}
\maketitle 
\begin{abstract}
We construct characteristic-zero lifts of partial Hasse invariants for genus zero non-compact curves in Hilbert modular varieties.
The construction is based on recent results on the associated Picard-Fuchs differential equations. 
As an application, we relate the size of the non-ordinary locus of the modulo~$p$ reduction of these curves to the dimension of spaces of (twisted) modular forms. We compute it explicitly for several Teichmüller curves, obtaining Deuring-like formulae.  
Moreover, we study the modulo~$p$ reduction of (twisted) modular forms on not necessarily arithmetic genus-zero Fuchsian groups with modular embedding. 
\end{abstract}
The systematic study of supersingular elliptic curves was initiated by M. Deuring~\cite{Deu} in~1931, who made explicit the relation between maximal orders in certain quaternion algebras and endomorphism ring of elliptic curves (over~$\overline{\mathbb{F}}_p)$ with trivial~$p$-torsion. This, together with Eichler's mass formula, gives the following formula for the number of supersingular~$j$-invariants over~$\overline{\mathbb{F}}_p$
\begin{equation}
\label{eq:Deuring}
\sharp\{\text{supersingular }j\text{-invariants over }\overline{\mathbb{F}}_p\}\=\Bigl\lfloor{\frac{p}{12}}\Bigr\rfloor
\+\delta\+\epsilon\,,
\end{equation}
where~$\delta=1$ if~$p\equiv 2 \mod 3$ and~$\delta=0$ otherwise, and~$\epsilon=1$ if~$p\equiv 3\mod 4$ and~$\epsilon=0$ otherwise. 

The supersingular locus of the moduli space of elliptic curves is strictly related to the theory of modular forms in characteristic~$p>0$. At an elementary level, this can be roughly described as follows. Let~$k$ be an extension of~$\F_p$ and consider the assignment 
\[
H_p: (E\,,\,y^2=f(x))\mapsto a_{p-1}\,,\quad \text{where }f(x)^{\frac{p-1}{2}}=\sum_{n\ge0}{a_nx^n}\,,
\]
that to an elliptic curve~$E/k$ with Weierstrass model~$y^2=f(x)$ associates the coefficient of~$x^{p-1}$ in the expansion of~$f(x)^{(p-1)/2}$. This rule defines a modular form~$H_p$ over~$k$, called~\emph{Hasse invariant}, which is of fundamental importance in the theory of modular forms in positive characteristic. It has the following properties: the divisor of~$H_p$ is supported on the supersingular locus, i.e. $a_{p-1}=0$ if and only if~$y^2=f(x)$ is a supersingular elliptic curve; moreover, the~$q$-expansion of~$H_p$ at the cusp is constant and equal to~$1$. 
The modular form~$H_p$ can be lifted to modular forms in characteristic zero. A choice of lift of~$H_p$ is the Eisenstein series~$E_{p-1}$ of weight~$p-1$. In particular, the classical description of~$E_{k}=1-\frac{2k}{B_k}\sum_{n\ge1}\sigma_{k-1}(n)q^n$ and the properties of Bernoulli numbers immediately imply that~$E_{p-1}\equiv 1\mod p$. 

In the Hilbert setting, the Hasse invariant is naturally replaced by a collection of modular forms called~\emph{partial Hasse invariants}, first introduced by Goren~\cite{G2} and studied in more detail in the memoir~\cite{AG} by Andreatta and Goren. 
Let~$F$ be a totally real field of degree~$g$ over~$\Q$ and let~$\mathcal{M}_{F,s}$ be the associated Hilbert modular variety over~$s=\mathrm{Spec}(\overline{\F}_p)$. When~$p$ is unramified in~$F$, there exist~$g$ Hilbert modular forms~$H_{p,1},\dots,H_{p,g}$ over~$\overline{\F}_p$ (actually over a finite field), of non-parallel weight~$p-1$, with the following properties: first, the divisor of~$H_{p,j}$ is a component of the non-ordinary locus in~$\mathcal{M}_{F,s}$, and the non-ordinary locus is precisely the union of such components; second, the~$q$-expansion of~$H_{p,j}$ is constant at every cusp. See Section~\ref{sec:charp} for details.
Contrary to the elliptic case though, the partial Hasse invariant do not lift to Hilbert modular forms in characteristic zero: they are purely characteristic~$p$ objects.

The main aim of this paper is to show that, when restricted to certain curves in Hilbert modular varieties, (powers of) partial Hasse invariants lift to modular forms in characteristic zero. 
Let~$Y\hookrightarrow M_F=\Po^g/\mathrm{SL}_2(\mathcal{O}_F)$ be an affine algebraic smooth curve embedded in a Hilbert modular variety (these curves are examples of~\emph{affine Kobayashi geodesics}). It follows from a result of Möller and Viehweg~\cite{MV} that~$Y$ can be defined over a number field $K$, and therefore can be equipped with an integral model over~$\mathcal{O}_K[S^{-1}]$ by taking the Zariski closure in the integral structure of the Hilbert modular variety (see Section~\ref{sec:phpol}). Here~$\mathcal{O}_K$ is the ring of integers in~$K$ and~$S$ is a finite set of primes. 
By modular form above we mean~\emph{twisted modular form}, as introduced by Möller and Zagier in~\cite{MZ}. If~$Y\simeq\Po/\Gamma\hookrightarrow M_F$ is a curve, a twisted modular form on~$\Gamma$ is a holomorphic function~$f\colon\Po\to\C$ with multi-weight~$\vec{k}=(k_1,\dots,k_g)$ whose transformation formula depends on the embedding of~$Y$ in~${M}_F$. See Section~\ref{sec:intro1} for the precise definition. We denote by~$M_{\vec{k}}(\Gamma,\varphi)$ the space of twisted modular forms of multi-weight~$\vec{k}$.

In this paper, we construct liftings of powers of partial Hasse invariants for every genus zero curve in a Hilbert modular variety of dimension~$g\ge2$ with no restriction on the base field~$K$, and for almost every prime unramified in the field of real multiplication and in~$K$ (Theorem~\ref{thm:phi}). Under the assumption~\eqref{eq:ass}, which concerns elliptic points and is satisfied in infinitely many cases (in particular by the torsion-free curves), we construct liftings of partial Hasse invariants without the need of taking powers (Theorem~\ref{thm:phi2}). 

For ease of exposition, we state the results in the introduction in the case of~$Y$ torsion-free, $K=\Q$, $g=2$, and~$p$ inert in~$F$.
Let~$\mathcal{X}\to\mathcal{Y}$ be the family of abelian surfaces giving an integral model to~$Y\hookrightarrow M_F$. For a prime~$p\not\in S$, inert in~$F$, denote by~$\mathcal{X}_p\to\mathcal{Y}_p$ the modulo~$p$ reduction of this model and by $\mathcal{X}_{\bar{p}}=\mathcal{X}_p\otimes_{\F_p}\overline{\F}_{p}$ its base change to~$\overline{\F}_{p}$.

As mentioned above, the non-ordinary locus of the Hilbert modular surface~$\mathcal{M}_{F,s}$ has two components~$D_1$ and~$D_2$. The intersection~$\mathcal{X}_{\bar{p}}\cap(D_1\cup D_2)$ is finite for generic~$p$, and in the genus-zero case its points can be identified with the zeros of polynomials~$\mathrm{ph}_{p,j}(t)\in {\F}_p[t]$ for~$j=1,2$ defined in Equation \eqref{eq:phpdef} (see Section~\ref{sec:intro1}).

\begin{theorem*}[Theorem~\ref{thm:phi}, Part 1]
Let~$\Po/\Gamma\simeq Y\hookrightarrow M_F$ and~$p$ be as above. 
There exists twisted modular forms~$h_{p,1}$ and~$h_{p,2}$ on~$\Gamma$, of weight~$(-1,p)$ and~$(p,-1)$ respectively, and a Hauptmodul~$t$ such that the polynomial in~$t$ describing the zeros of~$h_{p,j}$ in a fundamental domain of~$\Gamma$ has~$p$-integral coefficients and reduces to~$\mathrm{ph}_{p,j}(t)$ modulo~$p$ for~$j=1,2$.
\end{theorem*}

This theorem states that the twisted modular forms~$h_{p,j}$ satisfy one of the properties of the partial Hasse invariants, relative to the non-ordinary locus of~$\mathcal{X}_{\bar{p}}$. The other property, related to the~$q$-expansion of the partial Hasse invariants, is more subtle. 
The problem is that the uniformizing group of a Kobayashi curve~$Y$ is not an arithmetic Fuchsian group in general. The associated (twisted) modular form may not have algebraic Fourier coefficients at the cusps. In some cases it has been shown  (see~\cite{Wo} for the case of triangle curves, or~\cite{MV} for the case of Teichmuller curves in Hilbert modular surfaces) that, if suitably normalized, the Fourier coefficients can be made algebraic, but in general not~$p$-integral for a given prime~$p$. 
Nevertheless, it has been shown~\cite{B} that the local expansion at the cusp of a twisted modular form in terms of a Hauptmodul~$t$ is better behaved. More precisely, let $t$ be a Hauptmodul normalized to be zero at a cusp $c$ of~$\Gamma$ and whose values at the other cusps is in $\mathcal{O}_{K}[S^{-1}]^{\times}$. Then for every weight~$(k_1,k_2)\in\Z^2$ there exists a basis of twisted modular forms~$M_{(k_1,k_2)}(\Gamma,\varphi)$ whose~$t$-expansions at $c$ have coefficients in~$\mathcal{O}_K[S^{-1}]$. We can then prove that the~$t$-expansions of the modular forms~$h_{p,j}$ have the desired property modulo~$p$. 
\begin{theorem*}[Theorem~\ref{thm:phi}, Part 2]
  Let~$h_{p,1}$ and~$h_{p,2}$ be the twisted modular forms in the previous theorem, and~$t$ an Hauptmodul as above. Then the~$t$-expansion of~$h_{p,j}$ is constant modulo~$p$.
\end{theorem*}

The proof of Theorem~\ref{thm:phi} is based on results of~\cite{B} on the integrality of the solutions of the Picard-Fuchs differential equations attached to the Kobayashi curve~$Y$, and on the congruences they satisfy; the results are recalled in~Section~\ref{sec:intro1}. These Picard-Fuchs differential equations are defined over~$\pro^1(\C)$, and because of this we restrict to the case of genus zero curves. The possibility of lifting restrictions of partial Hasse invariants to curves is likely not limited to the genus zero case though, and this will be investigated in the future.

We use the results above to deduce some properties of the modulo~$p$ reduction of more general twisted modular forms with~$p$-integral~$t$-expansions, and to compute the cardinality of the intersections~$\mathcal{X}_{\bar{p}}\cap D_j$, that is, the degree of~$\mathrm{ph}_{p,j}(t)$. In the first application, we find that the~$t$-expansion of a twisted modular form of suitable weight reduces modulo~$p$ to a rational function with poles on the non-ordinary locus of~$Y$.
\begin{corollary*}[Corollary~\ref{cor:modp}] 
Let~$p\not \in S$ be a prime unramified in~$F$. Let~$g$ be a twisted modular form of weight~$(-k+lp,-l+kp)$ for~$k,l\in\Z_{\ge0}$ with~$p$-integral~$t$-expansion. The following holds for the $t$-expansion of~$g$ 
\[
g(t)\equiv\frac{P_g(t)}{R(t)}\mod p\,,
\]
where~$P_g(t)$ is a polynomial which depends on~$g$, and~$R(t)$ is a polynomial with zeros on the non-ordinary locus which only depends on~$k$ and~$l$.
\end{corollary*}

In the second application we prove that, under the assumption~\eqref{eq:ass} on the order of the elliptic points, the the degree of~$\mathrm{ph}_{p,j}(t)$ is given by the dimension of a space of twisted modular forms. We state here a version of the result in the torsion-free case and for~$p$ inert in~$F$ (the statement is different in the presence of elliptic points).
\begin{corollary*}[Corollary~\ref{cor:degree}]
Let~$\mathcal{X}\to\mathcal{Y}$ be the integral model associated to a torsion-free genus zero curve in a Hilbert modular surface, and let~$\mathrm{ph}_{p,j}(t)$ be the polynomial describing the component of the non-ordinary locus~$\mathcal{X}_{\bar{p}}\cap D_j$ for~$j=1,2$. Then
\[
\deg(\mathrm{ph}_{p,j})\=\begin{cases}
\dim M_{(-1,p)}(\Gamma,\varphi)\-1\quad&\text{if }j=1\,,\\
\dim M_{(p,-1)}(\Gamma,\varphi)\-1\quad&\text{if }j=2\,,\\
\end{cases}
\]
where~$M_{(k_1,k_2)}(\Gamma,\varphi)$ is the space of twisted modular forms of weight~$(k_1,k_2)$ for~$(\Gamma,\varphi)$ over~$\C$.
\end{corollary*}
By using this corollary we obtain analogues of Deuring's formula~\eqref{eq:Deuring} for the intersection numbers of~$\mathcal{X}_{\bar{p}}$ and the components of the non-ordinary locus. We do it explicitly for many Teichmüller curves~$W_D$ in a Hilbert modular surface of discriminant~$D$ in Section~\ref{sec:example}. As a concrete example, for~$D=13$ and~$p\neq 2,3,13$ inert in~$\Q(\sqrt{13})$ the analogues of Deuring's formula for~$W_{13,\bar{p}}$ are
\begin{equation}
\label{eq:W13}
\deg(\mathrm{ph}_{p,1}(t))\=\Bigl\lfloor\frac{p-3}{4}\Bigl\rfloor\+\epsilon\,,\quad \deg(\mathrm{ph}_{p,2}(t))\=\Bigl\lfloor\frac{3p-1}{4}\Bigl\rfloor\+\epsilon\,,
\end{equation}
where~$\epsilon=1$ if~$p\equiv 1 \mod 4$ and~$\epsilon=0$ otherwise. The first values are displayed in the following table, also for~$p$ split in~$\Q(\sqrt{13})$ (the general formula for this case is in Proposition~\ref{prop:WD}): 
\[
\begin{array}{ccccccccccccc}
\toprule
p & 5 & 7 & 11 & 17 & 19 & 23 & 29 & 31 & 37 & 41 & 43 & 47\\
\midrule
\deg\mathrm{ph}_{p,1} & 1 & 1 & 2 & 12 & 4 & 17 & 21 & 7 & 9 & 10 & 32 & 11\\
\midrule
\deg\mathrm{ph}_{p,2} & 4 & 5 & 8 & 4 & 14 & 6 & 22 & 23 & 28 & 31 & 32 & 35\\
\bottomrule
\end{array}
\]
A result of Bachmat-Goren implies that~$D_1$ and~$D_2$ are the components of the supersingular locus for~$p$ inert in~$F$ (see the beginning of Secion~\ref{sec:example} for the complete statement). Therefore the formulae~\eqref{eq:W13} count the points of supersingular reduction of~$W_{13}$  for a prime~$p$ inert in~$\Q(\sqrt{13})$. One deduces the following bounds for the size of the supersingular locus of~$W_{13,p}$:
\[
\Bigl\lfloor\frac{3p-1}{4}\Bigl\rfloor\+\epsilon\le \sharp(W_{13,p}\cap (D_1\cup D_2))\le p-1\+\epsilon\,.
\]
The proof for more general~$D$ and the analogous result for~$p$ split can be found in Proposition~\ref{prop:WD} and Proposition~\ref{prop:W5} in Section~\ref{sec:example}.

In Section~\ref{sec:expl} we give an alternative construction (Theorem~\ref{thm:phi3}) of lifts of partial Hasse invariant for the triangle curve~$\Delta(2,5,\infty)\simeq W_5$. It exploits the existence of an explicit integral model, and the possibility to parametrize it in terms of modular forms. We use it to write down the Hasse-Witt matrix in terms of modular forms, and identify its entries with the lifts of the partial Hasse invariants. It is similar in spirit to the elementary description of the Hasse invariant for elliptic curves we gave at the beginning of the introduction. 

Finally, for the group~$\Delta(2,5,\infty)$ and the primes~$p=11,13$ we compute closed-form~$t$-expansions of lifts of partial Hasse invariants in terms of hypergeometric functions. For~$p=13$ we obtain the formulae
\[
h_{13,1}(\tau)^2\=(1-t)\cdot\frac{{}_2F_1\Bigl(\frac{9}{20},\frac{1}{20};1;t\Bigr)^{26}}{{}_2F_1\Bigl(\frac{7}{20},\frac{3}{20};1;t\Bigr)^2}\,,\quad
h_{13,2}(\tau)^2\=p_{13}(t)\cdot \frac{{}_2F_1\Bigl(\frac{7}{20},\frac{3}{20};1;t\Bigr)^{26}}{{}_2F_1\Bigl(\frac{9}{20},\frac{1}{20};1;t\Bigr)^2}\,,
\]
where~$p_{13}(t)=\frac{10249593282075625}{4194304} - \frac{10435551419195625}{4194304}t + \frac{729693733125}{16384}t^2 - \frac{51480625}{256}t^3$.
The procedure easily adapts to other primes.

\begin{Ack}
G.\,Bogo is funded by the Deutsche Forschungsgemeinschaft (DFG, German Research Foundation) — SFB-TRR 358/1 2023 — 491392403.
Y.\,Li is supported by the Heisenberg Program of the DFG, project number 539345613.
\end{Ack}

\section{Curves in Hilbert modular varieties}
\subsection{Partial Hasse invariants of Hilbert modular varieties}
\label{sec:charp}
Let~$F\subset\R$ be a totally real field of degree~$g$ over~$\Q$ and let~$\OO_F$ denote its ring of integers. Let~$R$ be a ring.
An~\emph{abelian variety over~$R$ with real multiplication by $\mathcal{O}_F$} is an abelian variety~$A\to\mathrm{Spec}(R)$ of dimension~$g$ together with a ring injection~$\iota\colon\mathcal{O}_F\to\mathrm{End}_R(A)$ such that the following condition holds. Let~$A^\vee$ be the dual abelian variety, and notice that~$\mathcal{O}_F$ acts on~$A^\vee$ by duality. Define 
\[
M_A:=\{\lambda\colon A\to A^t\,:\,\lambda\circ\iota(r)=\iota(r)^\vee\circ\lambda\,,\,\forall r\in\OO_F \}
\]
to be the~$\OO_F$-module of symmetric~$\OO_F$-linear homomorphisms from $A$ to $A^\vee$. We assume that the Deligne-Pappas condition~$A\otimes_{\OO_F}M_A\simeq A^\vee$ holds (this condition is always satisfied if~$\mathrm{char}(R)=0$). See Chapter 3 of  the book~\cite{GoBook} for details. 

For~$n\ge3$ there exists a fine moduli scheme~$\mathcal{M}_F^n\to\Spec(\Z[n^{-1}])$ classifying isomorphism classes of polarized abelian varieties with real multiplication by~$\OO_F$ and full level~$n$ structure. By forgetting the level structure one gets a scheme~$\mathcal{M}_F\to\Spec(\Z)$, which is a coarse moduli space for abelian varieties with real multiplication by~$\OO_F$. 
The scheme~$\mathcal{M}_F$ is the disjoint union of components~$\mathcal{M}_F(\fa)$, where~$\fa$ runs over the the strict class group~$Cl(F)^+$ of~$F$. For~$R=\C$ the component~$\mathcal{M}_F(\fa)(\C)$ corresponds to the quotient of~$\Po^g$ by the Hilbert modular group~$\SL(\OO_F\oplus\fa)$. This group acts on~$\Po^g$ via Möbius transformations:
\begin{equation}
\label{eq:action}
\SL(\mathcal{O}_F\oplus\fa)\ni\gamma=\begin{pmatrix}
a & b\\
c & d
\end{pmatrix}
\mapsto \gamma\cdot (z_1,z_2,\dots,z_g)\:=\biggl(\frac{az_1+b}{cz_1+d},\frac{a^{\sigma_2} z_2+b^{\sigma_2}}{c^{\sigma_2}z_2+d^{\sigma_2}},\dots,\frac{a^{\sigma_g} z_2+b^{\sigma_g}}{c^{\sigma_g} z_2+d^{\sigma_g}}\biggr)\,,
\end{equation}
where~$x^{\sigma_j}=\sigma_j(x)$, for~$j=1,\dots,j$, and $\sigma_1,\dots,\sigma_g$ denote the real embeddings of~$F$ (we set $\sigma_1=Id$).
In the following, we denote by~$M_F$ the complex points of~$\mathcal{M}_F$ and by~$\mathcal{M}_{F,s}$ the special fiber of~$\mathcal{M}_F$ over~$s=\mathrm{spec}(\overline{\F}_p)$ for~$p$ a prime unramified in~$F$.  

Let~$k$ be a field of characteristic~$p$ and let~$\overline{k}$ be an algebraic closure of~$k$. The~\emph{$p$-rank} of a~$g$-dimensional abelian variety~$A/k$ is the rank of the group of~$p$-torsion points of~$A(\overline{k})$. Recall that the abelian variety~$A/k$ is called~\emph{ordinary} if its~$p$-rank is maximal, i.e., if~$A[p](\overline{k})\simeq(\Z/p\Z)^g$; otherwise it is called~\emph{non-ordinary}. The abelian variety~$A/k$ is~\emph{supersingular} if it is isogenous over~$k$ to a product of~$g$ supersingular elliptic curves; if this isogeny is an isomorphism, one says that~$A/k$ is~\emph{superspecial}. 

The locus of ordinary abelian varieties in~$\mathcal{M}_{F,s}$ at~$s=\mathrm{Spec}({\overline{\F}_p})$ is a dense subscheme, and its complement is called the~\emph{non-ordinary locus}.
The non-ordinary locus in~$\mathcal{M}_{F,s}$ can be described in terms of divisors of certain Hilbert modular forms called~\emph{partial Hasse invariants}.
\begin{theorem*}[Goren~\cite{G2}, Andreatta-Goren~\cite{AG}]
\begin{enumerate}[wide=0pt]
Let~$F$ as above and~$p$ be a prime unramified in~$F$.
\item The non-ordinary locus in~$\mathcal{M}_{F,s}$ has~$g$ components, denoted by~$D_j$ for~$j=1,\dots,g$. There exists~$g$ Hilbert modular forms~$h_1,\dots,h_g$ of (non-parallel) weight~$p-1$, called~\emph{partial Hasse invariants}, such that~$\mathrm{div}(h_j)=D_j$. The partial Hasse invariants can be defined over a suitable finite field~$\F$.
\item For every~$j=1,\dots,g$, the~$q$-expansion of the partial Hasse invariant~$h_j$ at every cusp is constant and equal to~$1$.
\end{enumerate}
\end{theorem*}
Partial Hasse invariants are fundamental objects for the theory of Hilbert modular forms in positive characteristic, and share some properties with the classical Hasse invariant of elliptic curves. Nevertheless, it is important to emphasize that the partial Hasse invariants~$h_i$, if~$\dim{\mathcal{M}_F}>1$ , are genuinely characteristic~$p$ objects, i.e., they do not extend to the integral model $\mathcal{M}_F$.
The obstruction is intrinsic: each partial Hasse invariant~$h_i$ has non-parallel weight and is not a cusp form, and Hilbert modular forms with these properties simply do not exist in characteristic zero. By contrast, the product~$H=\prod_{j=1}^g{h_j}$ of all partial Hasse invariants, called the~\emph{total Hasse invariant}, lifts to a characteristic-zero Hilbert modular form of parallel weight (more precisely, some power of~$H$ lifts). Whether a lift is a Hilbert Eisenstein series of parallel weight~$E_{p-1}$, as happens in the one-dimensional case is ultimately related, in the case~$g=2$, to Leopold's conjecture (see the last section of~\cite{G2}).

\subsection{Affine Kobayashi geodesics and their non-ordinary locus}
\label{sec:phpol}
An algebraic curve~$Y\hookrightarrow A_g$ embedded in the space of polarized~$g$-dimensional abelian varieties is called a~\emph{Kobayashi geodesic} if the induced map of universal coverings~$\Po\to\Po_g$, which are identified with~$\Po$ and the Siegel upper-half space~$\Po_g$ respectively, is totally geodesic with respect to the Kobayashi metric. 

A theorem of Möller and Viehweg proves that, in the affine case, Kobayashi geodesics are curves in Hilbert modular varieties.
\begin{theorem*}[Möller-Viehweg~\cite{MV}]
Let~$Y\hookrightarrow A_g$ be an affine Kobayashi geodesic induced by a family~$f\colon X\to Y$ of polarized abelian varieties. If the general fiber of~$f$ is simple, there exists a totally real field~$F$ with~$[F:\Q]=g$ such that~$Y\hookrightarrow A_g$ factors through the Hilbert modular variety~$M_F$. 
Moreover, the embedding~$Y\hookrightarrow A_g$ can be defined over a number field~$K$. 
\end{theorem*}
It follows that a Kobayashi geodesic~$Y\hookrightarrow M_F$ defined over~$K$, can be endowed with an integral structure by taking its Zariski closure in the integral structure of~$\mathcal{M}_F$ discussed in Section~\ref{sec:charp}. The integral model is described by a family of polarized abelian varieties~$\mathcal{X}\to\mathcal{Y}$ with real multiplication by~$F$ defined over a ring
\begin{equation}
  \label{eq:R}
  R:=\mathcal{O}_K[S^{-1}],
\end{equation}
where~$\mathcal{O}_K$ is the ring of integers of~$K$, and~$S\subset\mathcal{O}_{K}$ is a finite set of primes.
We assume that if $\fp\in S$, then all its conjugates are in~$S$ or, equivalently, that $S$ comes from a set of rational primes. Then, via the map $\mathrm{Spec}(R)\to\mathrm{Spec}(\Z[S^{-1}])$, we can, and sometimes do, consider $\mathcal{X}$ as a scheme over $\Z[S^{-1}]$. 
A choice of model with minimal set~$S$ is currently unknown for generic Kobayashi curves (examples for some curves in Hilbert modular surfaces have been computed in Section 10 of \cite{BM}), and this is not our focus here. In the paper we assume that all the primes that ramify in~$R$ and in~$F$ are in~$S$, and make more specific assumptions later. 

In the case we consider, that of~$Y$ of genus zero and non compact, the associated integral model is of the form
\begin{equation}
\label{eq:model}
\mathcal{X}\to\spec\biggl(R\biggl[t,\frac{1}{\prod_{i=1}^s(t-t_i)}\biggr]\biggr)\,,
\end{equation}
for some~$t_1,\dots,t_s\in R$.

Fix a rational prime~$p\not\in S$ and an algebraic closure~$\overline{\F}_p$ of~$\F_p$.  Let~$\fp\subset\mathcal{O}_{K}$ be a prime over~$p$ and~$k_{\fp}:=\mathcal{O}_{K}/\fp\mathcal{O}_{K}$ the residue field. 

Consider the base change of the integral model~\eqref{eq:model} to the algebraic closure of the redisue field
\begin{equation}
\label{eq:modelp}
\mathcal{X}_{\bar{\fp}}:=\mathcal{X}\times_{\mathrm{spec}(R[t,\prod_{i=1}^s(t-t_i)^{-1}])}\spec\biggl(\overline{\F}_p\biggl[t,\frac{1}{\prod_{i=1}^s(t-t_i)}\biggr]\biggr)\,,
\end{equation}
where we assume that~$t_1,\dots,t_g$ stay distinct modulo~$\fp$. 
Denote by $\mathcal{X}_{\bar{\fp},t_{0}}$ the fiber of $\mathcal{X}_{\bar{\fp}}$ over $t=t_{0}\in\overline{\F}_p$.  
Since $\mathcal{X}_{\bar{\fp}}$ comes from reduction of a non-trivial family over $R$, for generic~$p$ only finitely many of its points are non-ordinary abelian varieties. Therefore, for such primes~$p$, the non-ordinary, supersingular, and superspecial locus of $\mathcal{X}_{\bar{\fp}}$  can be described via polynomials with coefficients in $\overline{\F}_p$:  
\[
\mathrm{no}_\fp(t):=\!\!\!\!\!\prod_{\overset{t_0\in\overline{\F}_p}{\mathcal{X}_{\bar{\fp},t_0}\text{non-ordinary}}}\!\!\!\!\!{(t-t_0)}\,,\quad
\mathrm{ss}_\fp(t):=\!\!\!\!\!\prod_{\overset{t_0\in\overline{\F}_p}{\mathcal{X}_{\bar{\fp},t_0}\text{supersingular}}}\!\!\!\!\!{(t-t_0)}\,,\quad
\mathrm{sp}_\fp(t):=\!\!\!\!\!\prod_{\overset{t_0\in\overline{\F}_p}{\mathcal{X}_{\bar{\fp},t_0}\text{superspecial}}}\!\!\!\!\!{(t-t_0)}\,.
\]
By means of the identification~$R\otimes\overline{\F}_p\simeq \prod_{\fp|p}\overline{\F}_p$, we define the~\emph{non-ordinary, supersingular}, and \emph{superspecial polynomials} associated to the modulo~$p$ reduction of~$Y$ respectively by 
\[
\mathrm{no}_p(t):=\bigl(\mathrm{no}_\fp(t)\bigr)_{\fp|p}\,,\quad \mathrm{ss}_p(t):=\bigl(\mathrm{ss}_\fp(t)\bigr)_{\fp|p}\,,\quad \mathrm{sp}_p(t):=\bigl(\mathrm{sp}_\fp(t)\bigr)_{\fp|p}\quad\in \prod_{{\fp|p}}\overline{\F}_p[t]\,.
\]

As recalled in Section~\ref{sec:charp}, the non-ordinary locus in~$\mathcal{M}_{F,s}$ is described by the union of divisors~$D_1,\dots,D_g$ of the partial Hasse invariants. For~$\fp\subset\mathcal{O}_K$ a prime ideal over~$p$, we refine the polynomial~$\mathrm{no}_{\fp}(t)$ by considering the fibers of the family~$\mathcal{X}_{\bar{\fp}}$ that lie on the divisor~$D_j$
\[
\mathrm{ph}_{\fp,j}(t):=\prod_{\overset{t_0\in\overline{\F}_p}{\mathcal{X}_{\bar{\fp},t_0}\in D_j}}(t-t_0)\,,\quad j=1,\dots,g\,.
\]
For~$j=1,\dots,g$, the~\emph{partial Hasse polynomial} 
\begin{equation}
\label{eq:phpdef}
\mathrm{ph}_{p,j}(t)\:=\bigl(\mathrm{ph}_{\fp,j}(t)\bigr)_{\fp|p}\quad\in \prod_{\fp|p}\overline{\F}_p[t]\,,
\end{equation}
is the polynomial describing the intersection of~$\mathcal{X}_{\bar{\fp}}$ with the divisor~$D_j$ of the partial Hasse invariant~$h_j$. 

The relation between the various polynomials defined above is recorded in the following lemma. Recall that the set~$\bigcap_{j=1}^g{D_j}$ is the superspecial locus in~$\mathcal{M}_{F,s}$.
\begin{lemma}
\label{lem:phpol}
It holds
\[
\mathrm{no}_p(t)=\mathrm{lcm}_{j=1,\dots,g}(\mathrm{ph}_{p,j}(t))\,,\quad \mathrm{sp}_p(t)=\mathrm{gcd}_{j=1,\dots,g}(\mathrm{ph}_{p,j}(t))\,,
\] 
where $\gcd$ and~$\mathrm{lcm}$ are intended component-wise. 
In particular, for~$g=2$ it holds
\[
\mathrm{ss}_p(t)=\begin{cases}
\mathrm{lcm}(\mathrm{ph}_{p,1}(t),\mathrm{ph}_{p,2}(t)) & \text{ if $p$ is inert in $F$}\,,\\
\mathrm{gcd}(\mathrm{ph}_{p,1}(t),\mathrm{ph}_{p,2}(t)) & \text{ if $p$ is split in $F$}\,.
\end{cases}
\]
\end{lemma}
\begin{proof}
We only need to prove the second statement. It follows from a result of Bachmat-Goren~\cite{G1}, which states that for~$g=2$ the supersingular locus coincides with the non-ordinary locus if~$p$ is inert in~$F$, and that the supersingular locus coincides with the superspecial locus if~$p$ is split in~$F$.
\end{proof}

\subsection{Twisted modular forms}
\label{sec:intro1}
Let~$\phi\colon \Po/\Gamma\simeq Y\hookrightarrow M_F$ be an affine Kobayashi geodesic. The existence of this embedding has two consequences at the level of universal covering spaces: first, the group~$\Gamma$ can be realized as a subgroup of~$\SL(\mathcal{O}_F\oplus\fa)$, for some~$\fa\in Cl(L)^+$, not of finite index if~$g>1$. In particular,~$\Gamma$ is defined over the totally real number field~$F$ and acts on~$\Po^g$ as in~\eqref{eq:action}. 
The second consequence is the existence of a holomorphic map~$\varphi=(\varphi_j)_{j=1}^g\colon\Po\to\Po^g$, called~\emph{modular embedding}, such that for every~$\gamma\in\Gamma$ and~$\tau\in\Po$
\begin{equation}
\label{eq:me}
\varphi_j(\gamma\tau)=\gamma^{\sigma_j}\varphi_j(\tau)\,\quad j=1,\dots,g\,,
\end{equation}
where~$\sigma_1,\dots,\sigma_g\colon F\to\R$ are the real embeddings, which act entry-wise on elements of~$\Gamma$.
It is possible to normalize~$\varphi$ in such a way that~$\varphi_1(\tau)=\tau$, and we adopt this normalization from now on. 

In the world of Fuchsian groups, those admitting a modular embedding are rare, and include as a subclass arithmetic Fuchsian groups and their subgroups. Examples in the case~$g=2$ are the uniformizing groups of modular curves, for which the modular embedding~$\varphi_2$ is simply a Möbius transformation, and the uniformizing group of Weierstraß-Teichmüller curves~\cite{MZ}, which are in general non-arithmetic Fuchsian groups. In the latter case (and in general), the modular embedding is given by a ratio of solutions of a second-order linear differential equation. For general~$g$, non-cocompact triangle groups are examples of non-arithmetic Fuchsian grups admitting a modular embedding (the differential equations involved in this case are of hypergeometric type~\cite{CW},~\cite{BN}).
An important remark is that not every Fuchsian group with totally real trace field admit a modular embedding~\cite{SW}. 

There is a notion of modular forms on the group~$\Gamma$ that takes into account the existence of the modular embedding~$\varphi=(\varphi_j)_{j=1}^g$. A~\emph{twisted modular form} of weight~$(k_1,\dots,k_g)$ associated to the pair~$(\Gamma,\varphi)$ is a holomorphic function~$f\colon\Po\to\C$ such that
\[
f(\gamma\tau)=f(\tau)\prod_{j=1}^g(c^{\sigma_j}\varphi_i(\tau)+d^{\sigma_j})^{k_j}\,,\quad\text{ for every }\gamma=\begin{pmatrix} a & b \\c & d\end{pmatrix}\in\Gamma\text{ and every }\tau\in\Po\,.
\]
Moreover~$f$ must satisfy the following growth requirement at the cusps: the function~$f_g(\tau)=\prod_{j=1}^g(c^{\sigma_j}\varphi_j(\tau)+d^{\sigma_j})^{-k_j}f(g\tau)$ must be bounded as~$\mathrm{Im}(z)\to\infty$ for every element~$g$ of the Hilbert modular group~$\SL(\OO_F\oplus\fa)$. We denote by~$M_{\vec{k}}(\Gamma,\varphi)$ the space of twisted modular forms of weight~$\vec{k}=(k_1,\dots,k_g)$ with respect to the modular embedding~$\varphi$. For fixed~$k$, it is a finite dimensional~$\C$-vector space. 

Twisted modular forms of weight~$(k,0,\dots,0)$ are classical modular forms. The restriction to~$\Gamma$ of a Hilbert modular form~$F(\tau_1,\dots,\tau_g)$ of weight~$(k_1,\dots,k_g)$, i.e., the function defined by~$f(\tau):=F(\varphi_1(\tau),\dots,\varphi_g(\tau))$, is a twisted modular form of weight~$(k_1,\dots,k_g)$. 
An important example of twisted modular form that does not come from classical or Hilbert modular forms is the derivative~$\varphi'_j$ of the~$j$-th component of the modular embedding~$\varphi_j$: it is a twisted modular form of weight~$(2,0,\dots,0,-2,0,\dots,0)$, with~$-2$ in the~$j$-th component; its transformation property is easily deduced from its definition~\eqref{eq:me}, while the boundedness at the cusps has been discussed in Proposition 1.1 of~\cite{MZ} in the case~$g=2$; the proof can be easily adapted to the general case. This example shows in particular the existence of twisted modular forms whose weight is negative in one component. 
The algebraic and differential structure of twisted modular forms have been determined for non co-compact triangle groups in~\cite{BN}; one finds in particular many twisted modular forms that are non-classical nor restrictions of Hilbert modular forms,  which are defined in terms of hypergeometric functions. 

Important numerical invariants associated to the pair~$(\Gamma,\varphi)$ are the~\emph{Lyapunov exponents}~$\lambda_1,\dots,\lambda_g$ defined by the identity
\begin{equation}
\label{eq:Le}
\deg(\mathrm{div}\varphi'_j)\=(\lambda_j-1)\chi(\Po/\Gamma)\,,\quad j=1,\dots,g\,,
\end{equation}
where~$\chi(\Po/\Gamma)$ is the orbifold Euler characteristic of~$\Po/\Gamma$ and~$\mathrm{div}\varphi_j'$ is the divisor of~$\varphi_j'$ in a fundamental domain of~$\Gamma$. This makes sense since~$\varphi_j'$ is a twisted modular form. One sees in particular that~$\lambda_1=1$ since~$\varphi_1(\tau)=\tau$. For equivalent definitions of Lyapunov exponents (in the case~$g=2$) see~ Section 1 of~\cite{MZ}.

Let~$\Gamma$ have genus zero. By studying the Picard-Fuchs differential equations associated to the Kobayashi geodesic $\Po/\Gamma\hookrightarrow\mathcal{M}_F$, one can prove the following result. 
\begin{theorem*}[\cite{B}]
Let~$(\Gamma,\varphi)$ be a non co-compact genus zero Fuchsian group with modular embedding~$\varphi$, and identify~$\Po/\Gamma$ with a punctured sphere~$Y\hookrightarrow\mathcal{M}_F$ defined over a number field~$K$ via the Haputmodul~$t$. Normalize~$t$ to have a zero at the cusp~$i\infty$. 
Let~$\mathcal{X}\to\mathcal{Y}$ be an integral model defined over~$R=\mathcal{O}_K[S^{-1}]$.  Then:
\begin{enumerate}
\item For every~$\vec{k}=(k_1,\dots,k_g)\in\Q\times\Z^{g-1}$, the space~$M_{\vec{k}}(\Gamma,\varphi)$ has a basis of twisted modular forms whose local expansion at~$i\infty$ in the parameter~$t$ has~$p$-integral coefficients for every prime~$p\not\in S$.
\item Let~$p\not\in S$ be a rational prime and $\fp$ a prime above $p$. Let~$N$ be the lowest common multiple of the order of the elliptic points of~$\Gamma$ (set~$N=1$ if $\Gamma$ has no torsion). For every~$j\in\{1,\dots,g\}$ there exists an index~$j'\in\{1,\dots,g\}$, depending on~$j$ and on the splitting of~$p$ in~$F$, and twisted modular forms~$F_j$ and~$F_{j'}$ such that for the~$t$-expansion of~$F_j$ and~$F_{j'}$ at~$i\infty$ it holds:
\begin{equation}
\label{eq:cong}
F_j(t)^N\equiv\alpha_{\fp,j}(t)^{l_jl_{j'}}F_{j'}(t^p)^N\mod \fp\,,
\end{equation}
Here~$\alpha_{\fp,j}(t)\in k_{\fp}[t]$, for~$k_\fp=\mathcal{O}_K/\fp$, is a polynomial with the same zero locus of~$\mathrm{ph}_{p,j}(t)$. The twisted modular forms~$F_j$ and~$F_{j'}$ have weight~$l_j>0$ in the~$j$-th component and weight~$l_{j'}$ in the~$j'$-th component respectively, and weight~$0$ else.  
\end{enumerate}
\end{theorem*}
\noindent
An example of the relation between the indices~$j$ and~$j'$ is given in Section~\ref{sec:example} in the case of real quadratic fields; the relation is completely described in Section~$1.5$ of~\cite{B}.
Points 1 and 2 of the theorem hold more generally for twisted modular forms with non-trivial multiplier system. The choice of cusp for the~$t$-expansion is also arbitrary, the statements hold at any cusp~$c$ for a suitable choice of Hauptmodul~$t$. 

The theorem implies in particular that the partial Hasse polynomials $\mathrm{ph}_{\fp,j}(t)$, which are defined over the field~$\overline{\F}_p$, can be described by the polynomials $\alpha_{\fp,j}(t) \in k_{\fp}[t]$.

The polynomials~$\alpha_{\fp,j}(t)$ are related to modulo~$\fp$ solutions of the Picard-Fuchs differential equations of~$Y$ and can be explicitly described in terms of the non-ordinary locus as follows (see~\cite{B} for details).
Let~$E_j\subset\C$ be the image, via the Hauptmodul~$t\colon\Po/\Gamma\to\C$, of the set
\[
\{\tau_0\in\Po\;|\; \tau_0\text{ is an elliptic point or }\varphi'_{j}(\tau_0)=0\}\,.
\]
The set~$E_j$ is finite and defined over~$R$ (it can be interpreted as a subset of the set of singular points of the Picard-Fuchs equations of~$Y$, which are defined over~$R$; see Section 1 of~\cite{B} for details on Picard-Fuchs differential equations).

Let~$(E_{j})_{\fp}$ denote the image of the set $E_{j}$ after reduction modulo~$\fp$. 
Let~$\mathcal{X}_{\bar{\fp}}$ be as in~\eqref{eq:modelp} and 
for~$j=1,\dots,g$ define the polynomial~$\mathrm{ph}_{p,j}^*(t)\in \prod_{\fp|p}\overline{\F}_p[t]$ by
\begin{equation}
\label{eq:polref}
\mathrm{ph}_{\fp,j}^*(t)\:=\prod_{\overset{t_0\in\overline{\F}_p\setminus (E_j)_\fp}{\mathcal{X}_{\bar{\fp},t_{0}}\in D_{j}}}(t-t_0)\,,\quad\mathrm{ph}_{p,j}^*(t):=\bigl(\mathrm{ph}_{\fp,j}^*(t)\bigr)_{\fp|p}\,.
\end{equation}
This polynomial is a refinement of the polynomial~$\mathrm{ph}_{p,j}(t)$ in~\eqref{eq:phpdef} and describes the intersection~$\mathcal{X}_{\bar{\fp}}\cap D_j$ outside the elliptic points of~$Y$ and the divisor of~$\varphi'_j$. 

For~$t_0\in E_j$ denote by~$(t_0)_\fp$ its modulo~$\fp$ reduction, and let~$n_0\in\Z_{\ge1}$ be the order of the stabilizer in~$\Gamma$ of its preimage~$\tau_0:=t^{-1}(t_0)\in\Po$. One has~$n_0>1$ only if~$\tau_0$ is an elliptic point of~$\Gamma$. Let~$N$ be the lowest common multiple of the orders of elliptic points of~$\Gamma$ and set~$N=1$ if there are no elliptic points. Then it holds (Corollary~2 in~\cite{B})
\begin{equation}
\label{eq:ajp}
\alpha_{p,j}(t):=\bigl(\alpha_{\fp,j}(t)\bigr)_{\fp|p}\= \Bigl(\mathrm{ph}_{\fp,j}^*(t)^N\prod_{t_0\in E_j}{\bigl((t_0)_\fp-t\bigr)^{N\cdot\epsilon^\fp_{j,0}/n_0}}\Bigl)_{\fp|p}\in (R\otimes\F_{p})[t]\,,
\end{equation}
where~$\epsilon^\fp_{j,0}\in\{0,1+\mathrm{ord}_{\tau_0}\varphi_j'\}$ satisfies~$\epsilon^\fp_{j,0}\neq 0$ only if~$\mathcal{X}_{\bar{\fp},(t_0)_\fp}$ is a non-ordinary abelian variety. The form of the exponent~$\epsilon^\fp_{j,0}$ also comes from the theory of differential equations.

Finally, Theorem~3 of~\cite{B} states that for every $\fp|p$ the degree of~$\alpha_{\fp,j}$ is 
\begin{equation}
\label{eq:degapj}
\deg(\alpha_{\fp,j}(t))=\frac{-\chi(Y)}{2}\cdot N\cdot(p\lambda_{j'}-\lambda_j)\,,
\end{equation}
where~$j$ and~$j'$ are related as in the Theorem. 

\section{Twisted modular forms modulo~$p$}
\subsection{Twisted modular forms on genus zero groups}
\label{sec:genus0}
We start by describing the ring of modular forms on genus zero Fuchsian groups. 
\begin{lemma}
\label{lem:D}
Let~$\Gamma$ be a non co-compact genus zero Fuchsian group, and let~$i\infty$ be a cusp. There exists a modular form~$\Delta=\Delta_{i\infty}$ with non-trivial multiplier system~$v_\Delta$ satisfying
\begin{equation}
\label{eq:Delta}
\mathrm{weight}(\Delta)\=\frac{-2}{\chi(\Po/\Gamma)}\,,\quad \mathrm{div}(\Delta)=1\cdot i\infty\,,
\end{equation}
where~$\chi(\Po/\Gamma)$ denotes the orbifold Euler characteristic. 
Moreover, if~$\Po/\Gamma$ admits a modular embedding defined over~$\mathcal{O}_K[S^{-1}]$, the modular form~$\Delta$ can be chosen to have~$p$-integral~$t$-expansion at the cusps, for almost every prime~$p\not\in S$. 
\end{lemma} 
\begin{proof}
Let~$t$ be a Hauptmodul for~$\Gamma$. The modular form~$\Delta$ can be constructed as an algebraic function of~$t$ and its derivative~$t'$ as shown below. 

Assume first that~$\Gamma$ has no elliptic points, and let~$c_1=i\infty,c_2,\dots,c_s$ be the cusps of~$\Gamma$. Note that~$s\ge3$, otherwise the Euler characteristic of~$\Po/\Gamma$ would be non-negative. Normalize the Hauptmodul~$t$ to have a pole at~$i\infty$, and a zero at the cusp~$c_2$. Then~$\frac{t'}{t}\prod_{i=3}^s{(t-t(c_i))^{-1}}$ is a weight-two modular form with a zero of order~$s-3=-\chi(\Po/\Gamma)$ at the cusp~$i\infty$, and no other zeros or poles on~$\overline{\Po}$. Then any~$(s-3)$-th root of this weight-two modular form gives a possible~$\Delta$. 

Let now~$\Gamma$ have~$r>0$ elliptic points~$e_1,\dots,e_r$, of order~$n_1,\dots,n_r$ respectively, and cusps~$c_1=i\infty,c_2,\dots,c_s$. Let~$N=\mathrm{lcm}(n_1,\dots,n_r)$. Normalize the Hauptmodul~$t$ to have the zero at~$i\infty$, and the pole at~$e_1$.
We construct modular forms~$Q_1,\dots,Q_r$ with non-trivial multiplier system~$v_{Q_1},\dots,v_{Q_r}$ as follows:
\begin{equation}
\label{eq:Qid}
\begin{aligned}
Q_1&\:=\left[\Bigl(\frac{t'}{t}\Bigr)^N\prod_{l=2}^r{\bigl(t-t(\tau_l)\bigr)^{-N(1-1/n_l)}}\prod_{m=2}^s\bigl(t-t(c_m)\bigr)^{-N}\right]^{\frac{1}{N\cdot n_1\cdot(-\chi(\Po/\Gamma))}}\\
Q_i&\:=\left[\Bigl(\frac{t'}{t}\Bigr)^N\bigl(t-t(\tau_i)\bigr)^{N\cdot(-\chi(\Po/\Gamma))}\prod_{l=2}^r{\bigl(t-t(\tau_l)\bigr)^{-N(1-1/n_l)}}\prod_{m=2}^s\bigl(t-t(c_m)\bigr)^{-N}\right]^{\frac{1}{N\cdot n_i\cdot(-\chi(\Po/\Gamma))}}\;i\ge2\,.
\end{aligned}
\end{equation}
If follows from the construction that~$Q_i$ satisfies
\begin{equation}
\label{eq:Qi}
\mathrm{weight}(Q_i)\=\frac{-2}{\chi(\Po/\Gamma)\cdot n_1}\,,\quad\mathrm{div}(Q_i)\=\frac{1}{n_1}\cdot e_i\,,
\end{equation}
and
\begin{equation}
\label{eq:relmf}
Q_i^{n_i}\=(t(e_i)-t)Q_1^{n_1}\,,
\end{equation}
for every~$i=1,\dots,r$. 
In this case, $\Delta$ is defined by~$\Delta\cdot t^{-1}=Q_1^{n_1}$.

The claim on the~$p$-integrality of the~$t$-expansion follows from the first part of theorem in Section~\ref{sec:intro1}.
\end{proof}

Consider now the case of a genus zero Fuchsian group~$\Gamma$ with modular embedding~$\varphi=(\varphi_j)_j\colon\Po\to\Po^g$, whose associated curve~$\Po/\Gamma$ has an integral model~$\mathcal{X}\to\mathcal{Y}$ over~$\mathcal{O}_K[S^{-1}]$ as in Section~\ref{sec:intro1}. One could construct twisted modular forms of prescribed divisor by considering~$t\cdot(\varphi_j')^{-1}$ instead of~$t'$ and arranging a bit the expressions in~\eqref{eq:Qid}, in order to take into account the zeros of~$\varphi_j'$. However, we do something different and more effective. 

\begin{proposition} 
\label{prop:Bj}
Let~$\Gamma$ be a non-cocompact genus zero Fuchsian group with modular embedding~$\varphi=(\varphi_j)_j\colon\Po\to\Po^g$. For every~$j=2,\dots,g$ there exists a twisted modular form~$B_j$, in general with non-trivial multiplier system, with the following properties:
\begin{itemize}
\item $B_j$ has no zeros nor poles on~$\overline{\Po}$;
\item the weight of~$B_j$ is~$(-\lambda_j,0,\dots,0,1,0,\dots,0)$, where~$1$ is in the~$j$-th position, and~$\lambda_j$ is as in~\eqref{eq:Le}.
\end{itemize}
\end{proposition}
\begin{proof}
  The twisted modular forms in the statement can be constructed as follows. Let~$j\in\{2,\dots,g\}$ and let~$s_j$ be a (classical) modular form on~$\Gamma$ such that~$\mathrm{div}(s_j)=\mathrm{div}(\varphi_j')$. The modular form~$s_j$ is given by a polynomial in~$Q_1,\dots,Q_r$ if~$\Gamma$ has elliptic points, or as a polynomial in~$t$ times a power of~$\Delta$ if~$\Gamma$ is torsion free. It follows that the twisted modular form~$s_j\cdot(\varphi_j')^{-1}$ has no roots nor poles in~$\overline{\Po}$. The same is true for the twisted modular form
\[B_j\:=\sqrt{s_j\cdot(\varphi'_j)^{-1}}\,,
\] 
where we consider the principal branch of the square root.

The classical relation between the degree of the divisor of a modular form and its weight gives~$\mathrm{weight}(s_j)=\frac{-2}{\chi(\Po/\Gamma)}\cdot\deg(\mathrm{div}(s_j))=\frac{-2}{\chi(\Po/\Gamma)}\cdot\deg(\mathrm{div}(\varphi'_j))=2(1-\lambda_j)$ by~\eqref{eq:Le}. Since the weight of~$\varphi_j'$ is~$(2,0,\dots,0,-2,0,\dots,0)$, where~$-2$ is in the~$j$-th position, the proposition follows. 
\end{proof}

One should think of multiplication by powers of~$B_j$, for different values of~$j$, as a way to reduce a twisted modular form to a classical modular form without changing its divisor (but the weight and the multiplier systems change). Similarly, by multiplying by powers of the~$B_j$, one gets twisted modular forms from classical modular forms. It is simple to prove that every twisted modular form, on a genus zero group, can be constructed in this way. 

\subsection{Lifts of powers of partial Hasse invariants}
\label{sec:phi}
Let~$\Gamma$ be a non-cocompact genus zero Fuchsian group with modular embedding~$\varphi\colon\Po\to\Po^g$, and identify~$\Po/\Gamma$ with a punctured sphere~$Y\hookrightarrow {M}_F$, defined over a number field~$K$, via a Hauptmodul~$t$.
We normalize~$t$ to have a zero at~$i\infty$ and a pole at a cusp, or at an elliptic point if~$\Gamma$ has only one cusp. This normalization is always possible because the orbifold Euler characteristic of~$Y$ must be negative. A negative Euler characteristic implies that~$\Gamma$ has at least three cusps if it is torsion-free; otherwise, if~$\Gamma$ has fewer than three cusps, it must contain elliptic points. Furthermore, the assumption of non-cocompactness ensures that there is always at least one cusp. 

For~$p\not\in S$ a prime number, let~$\alpha_{p,j}(t)\in(R\otimes\F_p)[t]$ be the polynomial in~\eqref{eq:ajp}, and let~$\widetilde{\alpha}_{p,j}\in\mathcal{O}_K[S^{-1}][t]$ be a lift of~$\alpha_{p,j}$ of degree~$d_{p,j}=\deg{\widetilde{\alpha}_{p,j}(t)}=\deg{\alpha_{p,j}(t)}$. 
Finally, for every~$j\in\{1,\dots,g\}$ let~$j'\in\{1,\dots,g\}$ be the unique index associated to~$j$ by the splitting behavior of~$p$ in the totally real field~$F$ as in Part 3 of the theorem in Section~\ref{sec:intro1}. Let~$N\in\Z_{\ge1}$ be as in~\eqref{eq:ajp}.

For~$j=1,\dots,g$ define the twisted modular form~$h^N_{p,j}(\tau)\in M_*(\Gamma,\varphi)$ by
\begin{equation}
\label{eq:hjp}
h^N_{p,j}(\tau)\:=\Bigl(\frac{\Delta}{t}\Bigr)^{d_{p,j}}B_j^{-N}B_{j'}^{p\cdot N}\cdot\widetilde{\alpha}_{p,j}(t)\,.
\end{equation}
The definition of~$h_{p,j}^N(\tau)$ depends on the choice of lift~$\widetilde{\alpha}_{p,j}(t)$, but not the properties we are interested in. 

\begin{theorem}
\label{thm:phi}
Let~$p\not\in S$ be an odd prime. For~$j=1,\dots,g$ the twisted modular form~$h^N_{p,j}(\tau)$ is holomorphic of (non-parallel) weight~$N\cdot(p-1)$, and of trivial multiplier system. Moreover~$h^N_{p,1}(\tau),\dots, h^N_{p,g}(\tau)$ satisfy:
\begin{enumerate}
\item The zero locus of~$h^N_{p,j}$ is described by a polynomial in~$t$ with~$p$-integral coefficients and such that its reduction modulo~$p$ has the same zero locus of~$\mathrm{ph}_{p,j}(t)$.
\item The~$t$-expansion of~$h^N_{p,j}(\tau)$ at every cusp is constant modulo~$p$. 
\end{enumerate}
\end{theorem}
\begin{proof}
The twisted modular forms~$h^N_{p,j}$ have been constructed precisely to satisfy point~$1$ of the statement. It remains to show that this construction implies the other claims of the theorem.

The weight of~$h^N_{p,j}$ is zero in almost every component and it is~$-N$ in the~$j$-th component and~$p\cdot N$ in the~$j'$-th component, as follows from the description of the weight of~$B_l$ in Proposition~\ref{prop:Bj}. In order to prove the statement, we only have to check that the weight in the fist component is zero. By Lemma~\ref{lem:D} and Proposition~\ref{prop:Bj} it follows that the weight in the first component is
\[
d_{p,j}\cdot\frac{-2}{\chi(\Po/\Gamma)}\+N\cdot\lambda_j-pN\cdot\lambda_{j'}\=0\,,
\]
since~$d_{p,j}=\frac{-\chi(\Po/\Gamma)}{2}\cdot N\cdot (p\lambda_{j'}-\lambda_j)$ as stated in~\eqref{eq:degapj}.

The computation of the multiplier system goes similarly. Let~$v_\Delta\colon\Gamma\to\C$ and~$v_{B_l}\colon\Gamma\to\C$ denote the multiplier systems of~$\Delta$ and of~$B_l$ ($l\ge2$) respectively. Let~$s_l\in M_*(\Gamma)$ be as in the proof of Proposition~\ref{prop:Bj} be characterized by the condition~$\mathrm{div}(s_l)=\mathrm{div}(\varphi_l')$. As mentioned in the proof of the proposition, $s_l$ is a polynomial in~$Q_1,\dots,Q_r$. Then, thanks to the relations~\eqref{eq:relmf} and the definition of~$\Delta$, we can write~$s_l=\Delta^{m_l}\prod_{i=1}^rQ_i^{\delta_{l,i}}\cdot \hat{s}_l(t)$, for some~$m_l,\delta_{l,i}\in\Z_{\ge0}$ and~$\widehat{s}_l(t)$ a polynomial of degree bounded by~$m_l$. More precisely,~$m_l$ is the degree of~$\mathrm{div}{\varphi_l'}$ outside the elliptic points. Since~$\varphi_l'$, as well as~$\widehat{s}_l(t)$, has trivial multiplier system, it follows that the multiplier system~$v_{B_l}$ satisfies
\begin{equation}
\label{eq:vB}
v_{B_l}^{2}\=v_{s_l}\=v_\Delta^{m_l}\prod_{i=1}^rv_{Q_i}^{\delta_{l,i}}\=v_\Delta^{m_l}\cdot v_{Q_1}^{n_1\sum_{i=1}^r{\frac{\delta_{l,i}}{n_i}}}\=v_\Delta^{\chi(\Po/\Gamma)(\lambda_l-1)}\,,
\end{equation}
since~$m_l+\sum_{i=1}^r{\delta_{l,i}/n_i}=\deg{\mathrm{div}\varphi_l'}=(\lambda_l-1)\chi(\Po/\Gamma)$ by~\eqref{eq:Le}.
The definition~\eqref{eq:hjp} of~$h^N_{p,j}$ implies then its multiplier system~$v_{p,j}$ is computed by
\[
v_{p,j}\=v_\Delta^{d_{p,j}}\cdot v_{B_j}^{-N}\cdot v_{B_{j'}}^{p\cdot N}\= v_\Delta^{d_{p,j}-\tfrac{N}{2}\cdot \chi(\Po/\Gamma)(\lambda_j-1)+\tfrac{pN}{2}\cdot\chi(\Po/\Gamma)(\lambda_{j'}-1)}\= v_\Delta^{\tfrac{\chi(\Po/\Gamma)}{2}N(p-1)}\,,
\]
where the last equality follows from~\eqref{eq:degapj}. The definition of~$\Delta$ and of~$Q_i$ imply that~$v_\Delta^{N\cdot\chi(\Po/\Gamma)}=1$ is the trivial multiplier system. It follows that~$v_{p,j}$ is trivial for $p$ odd. 

Finally, we discuss the~$t$-expansion of~$h^N_{p,j}$ at the cusps, whose properties essentially follow from the congruence~\eqref{eq:cong}. 
Theorem~3 in~\cite{B}, from which~\eqref{eq:cong} is deduced, concerns the integrality of coefficients of power series solutions of the~$g$ Picard-Fuchs differential equations associated to the model~$\mathcal{X}\to\mathcal{Y}$. These solutions, for~$j=1,\dots,g$, are the~$t$-expansions of a~$l_j$-th root of the modular form~$F_j$ in part~$3$ of the theorem in Section~\ref{sec:intro1} (recall that~$F_j$ has weight~$l_j$ in the $j$-th component and zero else). The proof of Theorem~3 in~\cite{B} uses a normalization of the differential operators (namely, that all finite regular singular points have a local exponent equal to~$0$ in the Riemann scheme) that, translated in the language of modular forms, implies that~$F_j$ has divisor supported at the point where the Hauptmodul~$t$ has its pole. 
In particular, the modular form~$F_jB_j^{-l_j}$, which has the same divisor of~$F_j$ and weight~$(l_j\lambda_j,0,\dots,0)$, is a power of~$\Delta/t$, more precisely~$F_jB_j^{-l_j}=(\Delta\cdot t^{-1})^{-l_j\lambda_j\chi(Y)/2}$, for every~$j=1,\dots,g$. 
Congruence~\eqref{eq:cong} implies then 
\[
\begin{aligned}
1&\;\equiv\;\alpha_{p,j}(t)^{l_jl_{j'}}\cdot F_{j'}(t)^{Np}F_j(t)^{-N}\;\equiv\;\widetilde{\alpha}_{p,j}(t)^{l_jl_{j'}}\cdot \Bigl(\frac{\Delta(t)}{t}\Bigr)^{Nl_jl_{j'}\frac{-\chi(Y)}{2}(p\lambda_{j'}-\lambda_j)}B_{j'}(t)^{Npl_jl_{j'}}B_{j}(t)^{-Nl_jl_{j'}}\\ &\= h^N_{p,j}(t)^{l_jl_{j'}}\mod p\,.
\end{aligned}
\]

It follows that the modulo~$p$ reduction of the~$t$-expansion of~$h_{p,j}$ at the cusp~$i\infty$ is constant, since its power is~$1$. 
This holds at every cusp, since the congruence~\eqref{eq:cong} holds at every cusp~$c_i$ with~$t$ replaced by the Hauptmodul~$t-t(c_i)$ with a simple zero in~$c_i$.
\end{proof}

\begin{remark}
Notice that the two main properties of~$h^N_{p,j}$ in Theorem~\ref{thm:phi} are a priori unrelated. For instance, the twisted modular form~$(\Delta/t)^{d_{p,j}}\widehat{\alpha}_{p,j}(t)\prod_{j=2}^g{B_j^{\beta_j}}$ satisfies the first property of the theorem for any choice integers~$\beta_1,\dots,\beta_g$, but not the second one in general. However, once the correct weight is fixed (according to the splitting of~$p$ in~$F$), i.e., the values of~$\beta_1,\dots,\beta_g$, the second property follows from the first one. Corollary~\ref{cor:modp} discusses the other implication.
\end{remark}

\subsection{Lifts of partial Hasse invariants}
Let~$\Gamma$ be a non-co-compact genus zero Fuchsian group. In this section we work under an extra assumption on the order of the elliptic points of~$\Gamma$, which will permit us to construct lifts of partial Hasse invariants of minimal weight, refining Theorem~\ref{thm:phi}. For every weight~$q\in\Q$ such that~$M_q(\Gamma,v)\neq\emptyset$, for some multiplier system~$v$, we assume that the number~$-\chi(\Po/\Gamma)\cdot q/2$ satisfies
\begin{equation}
\label{eq:ass}
-\frac{\chi(\Po/\Gamma)}{2}\cdot q\=m\+\sum_{i=1}^r\frac{b_{i}}{n_i}\quad\text{for unique }m\in\Z_{\ge0}\text{ and }b_{i}\in\{0,\dots,n_i-1\}\,.
\end{equation}
This situation holds in infinitely many cases: trivially if~$\Gamma$ is torsion free or has only one elliptic points, and in the case~$\Gamma=\Delta(n,m,\infty)$ is an affine triangle group with~$n$ and~$m$ coprime.

In the case~$\Gamma$ is the uniformizing group of a Kobayashi geodesic~$Y$ with model~$\mathcal{X}\to\mathcal{Y}$ defined over~$R=\mathcal{O}_K[S^{-1}]$, the above assumption~\eqref{eq:ass} gives us control on the non-ordinary reduction modulo~$p\not\in S$ of elliptic points of~$Y$ as follows.  By definition~\eqref{eq:ajp}, the polynomial~$\alpha_{p,j}(t)$ is a tuple of polynomials over finite fields, each one of degree~$-\frac{\chi(Y)}{2}N(p\lambda_{j'}-\lambda_j)$. Assumption~\eqref{eq:ass} applied to~$q=N\cdot(p\lambda_{j'}-\lambda_j)$ then implies that the exponent~$\epsilon_{j,i}^\fp$ in~\eqref{eq:ajp}, if~$t_i=t(e_i)\in R$ is the image of an elliptic point, does not depend on~$\fp$, but only on~$p$ (in virtue of the unique decomposition of~$\frac{\chi(\Po/\Gamma)}{2}\cdot q$). in other words, if an elliptic point of~$Y$ has non-ordinary reduction modulo one prime ideal~$\fp$ over~$p$, it has non-ordinary reduction for every prime ideal dividing~$p$. 

Recall that~$\mathrm{ph}_{p,j}^*(t)$ in~\eqref{eq:polref} is the polynomial describing the non-ordinary locus outside the set~$E_j$, i.e., the image via~$t$ of elliptic points of~$Y$ and of points lying on the divisor of~$\varphi_j'$. We consider now the non-ordinary locus outside the elliptic points alone. Let~$e_1,\dots, e_r$ be representatives of the elliptic points of~$\Gamma$ and define a new set~$E_j^e:=\{t(e_1),\dots,t(e_r)\}\subset E_j $. Denote by~$(E_j^e)_\fp$ its reduction modulo~$\fp$, and consider the polynomials
\begin{equation}
\label{eq:phipole}
\mathrm{ph}_{\fp,j}^e(t)\:=\prod_{\overset{t_0\in\overline{\F}_p\setminus (E_j^e)_\fp}{\mathcal{X}_{\bar{\fp},t_{0}}\in D_{j}}}(t-t_0)^{\epsilon_{j,0}}\,,\quad\mathrm{ph}_{p,j}^e(t):=\bigl(\mathrm{ph}_{\fp,j}^e(t)\bigr)_{\fp|p}\,.
\end{equation}
where~$\epsilon_{j,0}$ is as in~\eqref{eq:ajp}, but we suppressed the dependence on~$\fp$ since under Assumption~\eqref{eq:ass} it depends only on~$p$. Let~$d^e_{p,j}$ denote the degree of~$\mathrm{ph}^e_{p,j}(t)$, and let $\widetilde{\mathrm{ph}^{e}_{p,j}}(t)\in R[t]$ denote a lift of $\mathrm{ph}^e_{p,j}(t)$ of degree $d^{e}_{p,j}$. 

For~$j=1,\dots,g$ we define twisted modular forms~$h_{p,j}(\tau)$ by
\begin{equation}
\label{eq:hpj}
h_{p,j}(\tau)\:=\Bigl(\frac{\Delta}{t}\Bigr)^{d^e_{p,j}}\prod_{i=1}^rQ_i^{\epsilon_{j,i}}\cdot B_j^{-1}B_{j'}^{p}\cdot \widetilde{\mathrm{ph}^e_{p,j}}(t)\
\end{equation}
The twisted modular forms~$h_{p,j}(\tau)$, which depends on the choice of polynomial~$\widetilde{\mathrm{ph}_{p,j}^e}(t)$, are lifts of partial Hasse invariants of minimal weight and minimal divisor. 
\begin{theorem}
\label{thm:phi2}
Let~$p\not\in S$ be an odd prime. For~$j=1,\dots,g$ the twisted modular form~$h_{p,j}$ is holomorphic of (non-parallel) weight~$(p-1)$. It satisfies:
\begin{enumerate}
\item The zero locus of~$h_{p,j}$ outside the elliptic points is described by a polynomial in~$t$ with~$p$-integral coefficients and such that its reduction modulo~$p$ has the same zero locus of~$\mathrm{ph}^e_{p,j}(t)$. Moreover,~$h_{p,j}$ has a zero at an elliptic point~$e_i$ of~$\Gamma$ only if the related fiber~$\mathcal{X}_{(t_i)}$, where~$t_i=t(e_i)\in R$, has non-ordinary reduction modulo every~$\fp|p$. 
\item The~$t$-expansion of~$h_{p,j}$ at every cusp is constant modulo~$p$. 
\end{enumerate}
Moreover $(h_{p,j})^{2}$ has trivial multiplier system for every $p$, and $h_{p,j}$ has trivial multiplier system for every sufficiently big prime $p$. 
\end{theorem}
\begin{proof}
If~$N=1$, then the statement reduces to Theorem~\ref{thm:phi}, so we assume that~$Y$ has~$r\ge1$ elliptic points~$e_1,\dots,e_r$ of order~$n_1,\dots,n_r$ respectively.  
The identities~\eqref{eq:relmf} and the definition of~$\Delta=Q_1^{n_1}\cdot t$ imply that~$h_{p,j}(\tau)^N=h^N_{p,j}(\tau)$, since
\[
\prod_{i=1}^r{Q_i}^{\epsilon_{j,i}\cdot N}\=Q_1^{\epsilon_{j,1}\cdot N}\prod_{i=2}^rQ_1^{n_1\epsilon_{j,i}\cdot\frac{N}{n_i}}(t(e_i)-t)^{\epsilon_{j,i}\cdot\frac{N}{n_i}}\=\Bigl(\frac{\Delta}{t}\Bigr)^{N\cdot\sum_{i=1}^r\frac{\epsilon_{j,i}}{n_i}}\prod_{i=2}^r(t(e_i)-t)^{\epsilon_{j,i}\cdot \frac{N}{n_i}}\,,
\]
and~$d_{p,j}=d^e_{p,j}+N\sum_{i=1}^r\frac{\epsilon_{j,i}}{n_i}$, as one can see from~\eqref{eq:ajp}. Therefore, claims 1 and 2 of the theorem follow from the analog claims Theorem~\ref{thm:phi}, and only the claim on the multiplier systems is left to be proven. 

Fix~$m\in\Z_{\ge0}$, and~$\vec{b}=(b_1,\dots,b_r)\in\Z_{\ge0}^r$ with~$0\le b_i\le n_i-1$, and~$\vec{\beta}=(\beta_2,\dots,\beta_g)\in\Z^{g-1}$, and consider the following space of holomorphic twisted modular forms
\[
M_{m,\vec{b},\vec{\beta}}(\Gamma,\varphi)\=\mathrm{span}_\C\biggl\{\Bigl(\frac{\Delta}{t}\Bigr)^m\cdot\prod_{i=1}^rQ_i^{b_i}\cdot\prod_{j=2}^g{B_j^{\beta_j}}\cdot g(t)\;\bigl|\; g(t)\in(\mathcal{O}_K[S^{-1}])[t]\;,\deg{g(t)}\le m\biggr\}\,.
\]
Note that all elements in~$M_{m,\vec{b},\vec{\beta}}(\Gamma,\varphi)$ have the same weight and the same multiplier system, both determined by the choice of~$m$ and~$\vec{b}$ and~$\vec{\beta}$. Conversely, a choice of weight~$\vec{k}=(k_1,k_2,\dots,k_g)\in\Q\times\Z^{g-1}$ forces the vector~$\vec{\beta}$ to be~$\vec{\beta}=(k_2,\dots,k_g)$, and imposes the condition 
\[
k_1\+\sum_{j=2}^g{\lambda_gk_g}\=\frac{-2}{\chi(\Po/\Gamma)}\biggl(m+\sum_{i=1}^r\frac{b_i}{n_i}\biggr)\,,
\]
which in general does not determine~$m$ and~$\vec{b}=(b_1,\dots,b_r)$. However, under the assumption~\eqref{eq:ass}, this choice is unique, and the choice of weight~$\vec{k}$ determines the space~$M_{m,\vec{b},\vec{\beta}}(\Gamma,\varphi)$, and in particular a multiplier system, uniquely. We remark that if~$M_{\vec{k}}(\Gamma,\varphi)\neq\emptyset$, then it is always of the form~$M_{m,\vec{b},\vec{\beta}}(\Gamma,\varphi)$, because the divisors of the modular forms $Q_1,\dots,Q_r$ and $\Delta$ are minimal, and then are the building blocks of every modular form on~$\Gamma$ and, combined with~$B_2,\dots,B_g$, of every twisted modular form.

For~$\vec{k}\in\Z^{g}$ consider the space~$M_{\vec{k}}(\Gamma,\varphi)$ of twisted modular forms of trivial multiplier system. If it is not empty, it is of the form~$M_{m,\vec{b},\vec{\beta}}(\Gamma,\varphi)$ for some~$m,\vec{b},\vec{\beta}$. Under Assumption~\eqref{eq:ass} this space is uniquely determined, and in particular, its elements have trivial multiplier system. 

The twisted modular form~$h_{p,j}$ has weight~$\vec{k}_{p,j}=(0,\dots,0,-1,0\dots,0,p,0,\dots,0)$ with~$1$ and~$p$ in the $j$-th position and~$j'$-th position respectively. It has therefore trivial multiplier system, under Assumption~\eqref{eq:ass}, if there exists a twisted modular form of the same weight with trivial multiplier system. We notice such twisted modular form can be constructed, similarly to the proof of Lemma~\ref{lem:D}, for the weight~$2\cdot\vec{k}_{p,j}$ from~$\{t,t',\varphi_j',\varphi_{j'}'\}$, implying that~$h_{p,j}^2$ has trivial multiplier system. More generally, the existence of a twisted modular form of given weight can be shown with the Riemann-Roch theorem, as did in~\cite{MZ} in the case~$g=2$. In this case, to be able to estimate the dimension of the space of modular forms via the Riemann-Roch theorem, the prime~$p$ should be big enough with respect to the Lyapunov exponents of~$Y$ (for instance~$p\cdot\lambda_2\ge3$). In particular, the dimension is non zero for every prime~$p$ big enough.
\end{proof}

\begin{corollary}
\label{cor:degree}
Let~$\Gamma$ be a genus zero non-co-compact Fuchsian group with modular embedding~$\varphi$ and Lyapunov exponents~$\lambda_1,\dots,\lambda_g$. Assume that~$\Gamma$ satisfies~\eqref{eq:ass}. Let~$\Po/\Gamma\simeq Y\hookrightarrow M_F$ admit an integral model over~$\mathcal{O}_K[S^{-1}]$. Then, for every~$j=1,\dots,g$ and prime~$p\not\in S$ and~$\mathrm{ph}_{p,j}^e(t)$ as in~\eqref{eq:phipole} it holds
\[
\deg\mathrm{ph}_{p,j}^e(t)\=\dim M_{p\lambda_{j'}-\lambda_j}(\Gamma,v)\-1\,,
\]
where~$v$ is the multiplier system~$v=v_\Delta^{\frac{\chi(Y)}{2}(p-1)}$.
Moreover, if~$h_{p,j}$ has trivial multiplier system it holds
\[
\deg\mathrm{ph}_{p,j}^e(t)\=\dim M_{(0,\dots,(p)_{j'},0,\dots0,(-1)_j,0,\dots,0)}(\Gamma,\varphi)\-1\,,
\] 
where~$j$ and~$j'$ are related as in Section~\ref{sec:intro1}.
\end{corollary}
\begin{proof}
If~\eqref{eq:ass} holds, then~$M_{p\lambda_{j'}-\lambda_j}(\Gamma,v)=M_{m,\vec{b},\vec{\beta}}(\Gamma,\varphi)$ for some~$m\ge0$ and~$\vec{b}\in(\Z_{\ge0})^r$ and~$\vec{\beta}\in\Z^{g-1}$. Therefore~$\dim{M_{p\lambda_{j'}-\lambda_j}(\Gamma)}=m+1$, since the dimension of the latter space is the number of monomials in~$t$ of degree at most~$m$. From the proof of Theorem~\ref{thm:phi} it follows that~$h_{p,j}B_jB_{j'}^{-p}\in M_{p\lambda_{j'}-\lambda_j}(\Gamma,v)$, and by definition~\eqref{eq:hpj} that~$h_{p,j}B_jB_{j'}^{-p}=({\Delta}/{t})^{d^e_{p,j}}\prod_{i=1}^rQ_i^{\epsilon_{j,i}}\cdot \widetilde{\mathrm{ph}^e_{p,j}}(t)$, where~$d^e_{p,j}$ is the degree of~$\widetilde{\mathrm{ph}^e_{p,j}}(t)$. Then~$d^e_{p,j}=m$. 
The second identity for~$\deg\mathrm{ph}_j^e(t)$ follows from the last statement of Theorem~\ref{thm:phi2} and its proof. 
\end{proof}
\begin{remark}
The corollary gives a convenient way to compute the degree of~$\mathrm{ph}_j^e(t)$, by computing the dimension of a space of twisted modular forms, which is easily done or by using the explicit generators of Section~\ref{sec:genus0}, or from the Riemann-Roch theorem. The relation between~$j'$ and~$j$, which is needed for the identification of the weight, is explicitly described in Section 1.5~\cite{B}. See Section~\ref{sec:example} for examples.
\end{remark}
\begin{remark}
The corollary may be true without assumption~\eqref{eq:ass}, but this is not clear from our methods. It may be proved in the general form by a careful study of the modulo~$p$ reduction of the Picard-Fuchs equations associated to the integral model of~$Y$, following some results of Dwork (see Chapter 9 of~\cite{D}).
\end{remark}

\begin{corollary}
\label{cor:modp}
Let~$\Gamma$ be a genus zero non-co-compact Fuchsian group with modular embedding~$\varphi$ and assume that~$\Gamma$ satisfies~\eqref{eq:ass}. Let~$\Po/\Gamma\simeq Y\hookrightarrow M_F$ admit an integral model over~$\mathcal{O}_K[S^{-1}]$. Let~$p\not\in S$ be a prime. For every~$j\in\{1,\dots,g\}$ let~$j'\in\{1,\dots,g\}$ be the index associated to~$j$, as in the theorem in Section~\ref{sec:intro1}, depending on the splitting of~$p$ in~$F$. 

For~$k\in\Z_{\ge0}$ let~$f_j$ be a twisted modular form of weight~$-k$ and~$kp$ in the~$j$-th and~$j'$-th component respectively, and weight zero in the other components, and assume that~$f_j$ has~$p$-integral~$t$-expansion at~$i\infty$. If~$2|k$ or if~$p$ is sufficiently big, then the~$t$-expansion of~$f_j$ at the cusp~$i\infty$ satisfy: 
\begin{equation}
\label{eq:fjmodp}
f_j(t)\;\equiv\;\frac{H_{f_j}(t)}{\mathrm{ph}_{p,j}^e(t)^k\prod_{i=1}^r(t(e_i)-t)^{\epsilon_{j,i}a_i}}\mod p\,,
\end{equation}
where~$H_{f_j}(t)$ is a polynomial of degree bounded by~$\lfloor{\deg\mathrm{div}(f_j)}\rfloor=\bigl\lfloor{k\cdot\frac{\chi(Y)}{2}(\lambda_j-p\lambda_{j'})}\bigr\rfloor$, the exponents~$\epsilon_{j,i}$ are as in~\eqref{eq:hjp}, and~$a_i\in\Z_{\ge0}$.

Conversely, for every polynomial~$H_j(t)$ satisfying the above degree bound, there exists a twisted modular form~$g_j$ of weight~$-k$ in the~$j$-th component, $pk$ in the~$j'$-th component, and zero else, whose modulo~$p$ reduction has the above form with~$H_{g_j}(t)=H_j(t)$. 
\end{corollary}

\begin{remark}
 This result was proven by Koike (Proposition 5 in~\cite{Koike}) in the case~$\Gamma=\SL_2(\Z)$.
\end{remark}
\begin{remark}
A version of the corollary holds more generally for~$\Gamma$ not satisfying~\eqref{eq:ass}. In this case, the components of the weight of the modular form~$f_j$ must be divisible by~$N$, where~$N$ is the lowest common multiple of the orders of elliptic points of~$\Gamma$. In this case, the denominator of~\eqref{eq:fjmodp} is given by the polynomial~$\alpha_{p,j}(t)$. The proof of this fact is obtained from the one below, by replacing~$h_{p,j}(\tau)$ with~$h^N_{p,j}(\tau)$. 
\end{remark}
\begin{proof}
Let~$h_{p,j}$ be the partial Hasse invariant of weight~$-1$ in the~$j$-th component and~$p$ in the~$j'$-th component. By the assumption on~$k$ or~$p$, the power~$h_{p,j}(\tau)^k$ has trivial multiplier system, as follows from Theorem~\ref{thm:phi2}, and therefore the product~$f_j(\tau)\cdot h_{p,j}(\tau)^{-k}$ has weight zero and it is a rational function of the Hauptmodul~$t$. By writing the modular form~$f_jB_j^kB_{j'}^{-kp}$ in terms of~$\Delta$ and~$Q_1,\dots,Q_r$, and by the definition~\eqref{eq:hpj} of~$h_{p,j}$ we get
\begin{equation}
\label{eq:decj}
\begin{aligned}
\frac{f_j}{h_{p,j}^k}\=\frac{f_jB_j^kB_{j'}^{-kp}}{h_{p,j}^kB_j^kB_{j'}^{-kp}}\=\frac{(\tfrac{\Delta}{t})^{m}\prod_{i=1}^rQ_i^{b_i}\cdot g_j(t)}{(\tfrac{\Delta}{t})^{k\cdot d^e_{p,j}}\prod_{i=1}^rQ_i^{k\cdot\epsilon_{j,i}}\cdot\mathrm{ph}_{p,j}^e(t)^k}\=\Bigl({\frac{\Delta}{t}\Bigr)^{m-k\cdot d^e_{p,j}}}\prod_{i=1}^rQ_i^{b_i-k\cdot\epsilon_{j,i}}\frac{g_j(t)}{\mathrm{ph}^e_{p,j}(t)^k}\,,
\end{aligned}
\end{equation}
for some polynomial~$g_j(t)\in (\mathcal{O}_K[S^{-1}])[t]$. 
The modular form~$(\Delta/t)^{m-d_j\cdot k}\prod_{i=1}^rQ_i^{b_i-k\cdot\epsilon_{j,i}}$ has weight zero and is a rational function of~$t$. Its zeros and poles, which are located at the elliptic points that are zeros of~$f_j$ and~$h_{p,j}$ respectively, have therefore integral order. In particular the exponent~$b_i-k\epsilon_{j,i}$ of~$Q_i$  must be a multiple of the order~$n_i$ of the elliptic point~$e_i$. Since~$b_i\in\{0,\dots,n_i-1\}$ and~$\epsilon_{j,i}\ge0$ hold, it follows~$b_i-k\epsilon_{j,i}\le0$.
The first statement of the corollary follows therefore by setting~$H_{f_j}(t):=g_j(t)$ and~$a_i\in\Z\ge0$ such that~$a_i\epsilon_{j,i}=b_i-k\epsilon_{p.i}$. The degree of~$H_{f_j}$ is then bounded by~$m$, which by construction is bounded by the integral part of~$\deg\mathrm{div}(f_j)$.

Conversely, let the exponents~$m$ and~$b_i$ be as in~\eqref{eq:decj}. For every polynomial~$H_j(t)$ with~$p$-integral coefficients and degree bounded by~$m$, the twisted modular form~$(\Delta/t)^{m}\prod_{i=1}^rQ_i^{b_i}\cdot H_j(t)\cdot B_j^{-k}\cdot B_{j'}^{kp}$ has, by construction, multiplier system and weight as~$h_{p,j}^k$ and~$p$-integral $t$-expansion at~$i\infty$. Therefore the modulo~$p$ reduction of its~$t$-expansion is as in the statement of the corollary with~$H_{f_j}(t)=H_j(t)$.
\end{proof}

\subsection{Examples: curves in Hilbert modular surfaces}
\label{sec:example}
We give some examples of Corollary~\ref{cor:degree} in the case of curves in Hilbert modular surfaces associated to a real quadratic field~$F$. In this case the index~$j$ takes value in the set~$\{1,2\}$ and the relation between~$j$ and~$j'$ in Part 3 of Theorem~\ref{sec:intro1} is explicitly described as follows, in terms of the splitting of an unramified prime~$p\in\Z$:
\begin{itemize}
\item if~$p$ is split in~$F$, then it holds~$j=j'$ for~$j=1,2$. The partial Hasse invariants~$h_{p,1}$ and~$h_{p,2}$ have weight~$(p-1,0)$ and~$(0,p-1)$ respectively;
\item if~$p$ is inert in~$F$, then it holds~$j=2'$ if~$j=1$ and~$j'=1$ if~$j=2$. The partial Hasse invariants~$h_{p,1}$ and~$h_{p,2}$ have weight~$(-1,p)$ and~$(p,-1)$ respectively.
\end{itemize}
Moreover, in the case of curves in Hilbert modular surfaces, the degrees of the partial Hasse polynomials give information on the cardinality of the supersingular locus, as the following result of Bachmat and Goren shows.
\begin{theorem*}[Bachmat-Goren~\cite{G1}]
Let~$D$ be a positive integer, and let~$F=\Q(\sqrt{D})$. Let~$p$ be a prime unramified in~$F$. If~$p$ is inert in~$F$, then the non-ordinary locus and the supersingular locus in~$\mathcal{M}_{F,s}$ coincide. If~$p$ is split in~$F$, the supersingular locus and the superspecial locus in~$\mathcal{M}_{F,s}$ coincide (but they differ in general from the non-ordinary locus). 
\end{theorem*} 

In the computations below we make use of a formula of Möller-Zagier for the dimension of the space of twisted modular forms, which comes from a Riemann-Roch-type result. Similar computations can be made, in the genus zero case, by using the description of the space of modular forms given at the beginning of Section~\ref{sec:phi}.

\subsubsection*{Weierstrass-Teichmüller curves~$W_D$}
We consider the family of Teichmüller curves~$W_{D}$ in~$\mathcal{A}_2$. For the definition and the main properties we refer to Section~5.3 of~\cite{MZ}, from which we report here the information relevant to our discussion. For every non-square discriminant~$D$ the Hilbert modular surface~$\Po^2/\SL(\mathcal{O}_D^\vee\oplus\mathcal{O}_D)$ contains precisely one algebraically primitive Teichmüller curve~$W_D$ if~$D\not\equiv 1\mod 8$; if~$D\equiv 1\mod8$ there are two homeomorphic algebraically primitive Teichmüller curves. The curves~$W_D$ have only elliptic points of order~$2$ if~$D\ge5$, and in the case~$D=5$, treated in the next section, there is an elliptic point of order two and an elliptic point of order five. The number of elliptic points~$e_2(D)$ of~$W_D$ is known~\cite{Muk}. Finally, the Lyapunov exponents are~$\lambda_1=1$ and~$\lambda_2=1/3$ for every~$W_D$. 

According to Mukamel's classification~(see Appendix B of~\cite{Muk}), for
\begin{equation}
\label{eq:list}
D=5,8,9,12,13,16,17,20,24,25,33,36,37,40
\end{equation}
the curve~$W_D$ satisfies Assumption~\eqref{eq:ass}, more precisely, the genus of~$W_D$ is zero, and~$e_2(D)\le1$. 
In these cases the degree of the partial Hasse polynomials~$\mathrm{ph}_1(t)$ and~$\mathrm{ph}_2(t)$, i.e., the intersection number of the modulo~$p$ reduction of~$W_D$ and the components~$D_1$ and~$D_2$ of the non-ordinary locus of~$\mathcal{M}_{\Q(\sqrt{D}),s}$, can be computed with the help of Corollary~\ref{cor:degree}. From now on we concentrate on discriminants~$D>5$ belonging to the list~\eqref{eq:list} satisfying~$D\neq 1\mod 8$, for which~$W_D$ has only one component, in order to make the exposition smoother; the remaining cases can be handled similarly. 

Denote by~$\Gamma_D$ the uniformizing group of~$W_D$ and by~$\varphi$ the modular embedding. Theorem~5.5 of~\cite{MZ} states that for~$D\neq 1\mod 8$ and~~$k_1,k_2\in\Z_{\ge0}$ with~$k_1+k_2$ even it holds
\begin{equation}
\label{eq:dim}
\dim\mathrm{M}_{(k_1,k_2)}(\Gamma_D,\varphi)\-1\=-\frac{\chi(W_D)}{2}\Bigl(k_1+\frac{k_2}{3}\Bigr)-e_2(D)\Bigl\{\frac{-k_1+k_2}{4}\Bigr\}\,,
\end{equation}
where~$\{x\}\in [0,1)$ denotes the fractional part of~$x$ and~$e_2(D)$ is the number of elliptic points of~$W_D$. An analogous statement holds for~$D\equiv 1\mod 8$. 

Let now~$p\in\Z$ be a good prime for~$W_D$. By substituting the weight of the partial Hasse invariant~$h_{p,j}$ for~$(k_1,k_2)$ in~\eqref{eq:dim}, one obtains the degree of~$\mathrm{ph}_j^*(t)$, as stated in Corollary~\ref{cor:degree}. 
In order to complete the study of the non-ordinary locus of~$W_D$ modulo~$p$, if~$e_2(D)>0$, it remains to determine the reduction of the fiber over the elliptic point. 

Let~$\epsilon_j\in\{0,1\}$ be zero if~$W_D$ is torsion-free  or if the (unique) elliptic point~$e_1$ has ordinary reduction. Then it follows from Theorem~\ref{thm:phi2} that~$\epsilon=1$ only if the modular form~$Q_1$, with zero of order~$1/2$ in~$e_1$ and non-zero else, divides~$h_{p,j}(\tau)$. Since the degree of~$\mathrm{div}(h_{p,j})$ is given by the right-hand side of~\eqref{eq:dim} for~$(k_1,k_2)$ the weight of~$h_{p,j}$, it follows that 
\[
\frac{\epsilon_{j}}{2}\=e_2(D)\cdot\Bigl\{\frac{-k_1+k_2}{4}\Bigr\}\,.
\]
In particular, given the explicit description of the weight of~$h_{p,j}$ at the beginning of~\ref{sec:example} it follows that 
\begin{equation}
\label{eq:chspi}
\frac{\epsilon_{j}}{2}\=e_2(D)\Bigl\{\frac{(-1)^{j-1}(p+1)}{4}\Bigr\}\,,\quad j=1,2\,,
\end{equation}
if~$p$ is inert in~$\Q(\sqrt{D})$, and, if~$p$ is split in~$\Q(\sqrt{D})$, that
\begin{equation}
\label{eq:chsp}
\frac{\epsilon_{j}}{2}\=e_2(D)\Bigl\{\frac{(-1)^{j}(p-1)}{4}\Bigr\}\,,\quad j=1,2\,.
\end{equation}
 
By adding the value of~$\deg{\mathrm{ph}^e_{p,j}(t)}$ given by~\eqref{eq:dim} and the contribution of the elliptic point~\eqref{eq:chspi} or~\eqref{eq:chsp}, according to the splitting of~$p$ in~$\Q(\sqrt{D})$ , we obtain the degree of the partial Hasse polynomials for~$W_D$.
\begin{proposition}
\label{prop:WD}
Let~$D>5$ be one of the discriminants on the list~\eqref{eq:list}, and assume~$D\neq1\mod 8$. Let~$W_D$ be the unique primitive Teichmüller curve~$W_D\subset\Po^2/\SL(\mathcal{O}_D^\vee\oplus\mathcal{O}_D)$. 
\begin{enumerate}[wide=0pt]
\item Let~$p$ be inert in~$\Q(\sqrt{D})$. It holds
\[
\deg(\mathrm{ph}_1(t))=\begin{cases}
-\bigl(\frac{p}{3}-1\bigr)\frac{\chi(W_D)}{2}+\frac{e_2(D)}{2}&\text{if}\; p\equiv 1\mod 4\\
-\bigl(\frac{p}{3}-1\bigr)\frac{\chi(W_D)}{2}&\text{if}\; p\equiv 3\mod 4\,,
\end{cases}
\]
and
\[
\deg(\mathrm{ph}_2(t))=\begin{cases}
-\bigl(p-\frac{1}{3}\bigr)\frac{\chi(W_D)}{2}+\frac{e_2(D)}{2}&\text{if}\; p\equiv 1\mod 4\\
-\bigl(p-\frac{1}{3}\bigr)\frac{\chi(W_D)}{2}&\text{if}\; p\equiv 3\mod 4\,.
\end{cases}
\]
In particular, if~$e_2(D)=1$, the elliptic point has superspecial reduction modulo~$p$ if~$p\equiv 1\mod4$.
\item Let~$p$ be split in~$\Q(\sqrt{D})$. It holds
\[
\deg(\mathrm{ph}_1(t))=\begin{cases}
-\bigl(p-1\bigr)\frac{\chi(W_D)}{2}&\text{if}\; p\equiv 1\mod 4\\
-\bigl(p-1\bigr)\frac{\chi(W_D)}{2}+\frac{e_2(D)}{2}&\text{if}\; p\equiv 3\mod 4\,,
\end{cases}
\]
and
\[
\deg(\mathrm{ph}_2(t))=\begin{cases}
-\bigl(p-1\bigr)\frac{\chi(W_D)}{6}&\text{if}\; p\equiv 1\mod 4\\
-\bigl(p-1\bigr)\frac{\chi(W_D)}{6}+\frac{e_2(D)}{2}&\text{if}\; p\equiv 3\mod 4\,.
\end{cases}
\]
\end{enumerate}
\end{proposition}
In particular,  Lemma~\ref{lem:phpol} and Proposition~\ref{prop:WD} imply the following estimate on the cardinality of the supersingular locus of~$W_D$ modulo a prime~$p$ inert in~$\Q(\sqrt{D})$:
\begin{align*}
-\Bigl(p-\frac{1}{3}\Bigr)\frac{\chi(W_D)}{2}+\frac{e_2(D)}{2}\;\le\;\mathrm{deg}(\mathrm{ss}_p^{W_D})\;\le\; -\frac{4(p-1)}{3}\chi(W_D)\+e_2(D)\qquad & \text{if }p\equiv 1\mod 4\\
-\Bigl(p-\frac{1}{3}\Bigr)\frac{\chi(W_D)}{2}\;\le\;\mathrm{deg}(\mathrm{ss}_p^{W_D})\;\le\; -\frac{4(p-1)}{3}\chi(W_D)\qquad & \text{if }p\equiv 3\mod 4\,.
\end{align*}
The superspecial polynomials satisfies instead
\begin{align*}
e_2(D)\;\le\;\mathrm{deg}(\mathrm{sp}_p^{W_D})\;\le\;-\bigl(p-\frac{1}{3}\bigr)\frac{\chi(W_D)}{2}+\frac{e_2(D)}{2} \qquad & \text{if }p\equiv 1\mod 4\\
0\;\le\;\mathrm{deg}(\mathrm{sp}_p^{W_D})\;\le\;(p-\frac{1}{3}\bigr)\frac{\chi(W_D)}{2} \qquad & \text{if }p\equiv 3\mod 4\,.
\end{align*}
The values of $\chi(W_{D})$ for $D$ in \eqref{eq:list} can be found in the Appendix B of \cite{Muk}.
One recovers the example~\eqref{eq:W13} for~$D=13$ in the introduction by setting the values~$\chi(W_{13})=-\tfrac{3}{2}$ and~$e_2(13)=1$.
Analogous estimates can be computed in the split case.

\subsubsection*{The triangle curve~$W_5=\Delta(2,5,\infty)$}
An important family of curves in Hilbert modular varieties are the~\emph{triangle curves}. These are of the form~$\Po/\Delta(n,m,\infty)$ where~$\Delta(n,m,\infty)$ for~$n,m\in\Z_{\ge0}\cup\{\infty\}$, is the group of symmetries of a hyperbolic triangle of angles~$\{0,\pi/n,\pi/m\}$. The curve~$\Po/\Delta(n,m,\infty)$ has genus zero, one cusp, one elliptic point of order~$n$ and one elliptic point of order~$m$ (cusps if~$n=\infty$ or~$m=\infty$). Familiar cases include the arithmetic groups~$\Delta(2,3,\infty)\simeq\SL_2(\Z)$ and~$\Delta(\infty,\infty,\infty)\simeq\Gamma(2)$, but in general the groups~$\Delta(n,m,\infty)$ are non arithmetic. 
The Picard-Fuchs equations of triangle curves are of hypergeometric type; the realization of triangle curves as Teichmüller curves, as well as their Lyapunov exponents, have been studied in~\cite{BMt}. 

In this section we consider the triangle curve~$\Po/\Delta(2,5,\infty)$, which is embedded in a Hilbert modular surface of discriminant~$D=5$. There is an exceptional isomorphism~$\Po/\Delta(2,5,\infty)\simeq W_5$, where~$W_5$ is a member of the family of curves discussed in the previous section. 

From the description of the algebraic structure of the space of twisted modular forms on~$\Delta(2,5,\infty)$ (see Section~\ref{sec:expl} of this paper or the last section of~\cite{BN}) it follows that
\begin{equation}
\label{eq:dim5}
\dim M_{(k_1,k_2)}(\Delta(2,5,\infty),\varphi)\-1=\frac{3k_1+k_2}{20}-\Bigl\{\frac{-k_1+k_2}{4}\Bigr\}-\Bigl\{\frac{-k_1+3k_2}{10}\Bigr\}\,,\quad (k_1,k_2)\in\Z^2\,.
\end{equation}
Let~$p\in\Z$ be a prime of good reduction. As in the previous example, the degree of~$\mathrm{ph}_j^*(t)$ is given by the formula above when~$(k_1,k_2)$ is the weight of~$h_{p,j}$. The non-ordinary reduction of the elliptic points is also related to the vanishing of the fractional parts of~\eqref{eq:dim5} as in the previous example, but the degree of the complete partial Hasse polynomials can not be computed directly as before. More precisely, everything works fine for the elliptic point of order~two (whose non-ordinary reduction depends only on the class of~$p$ modulo~$4$), but there a more detailed analysis is required for the elliptic point of order five. This is because the non-trivial component~$\varphi_2$ of the modular embedding has a zero of multiplicity $\tfrac{1}{5}$ at the elliptic point of order~five. We illustrate it completely in a special case, and than state the general result. 

Recall that~$p$ is inert in~$\Q(\sqrt{5})$ if and only if~$p\equiv2,3\mod5$. We assume that~$p\equiv 3\mod 5$; then~$p=3+n\cdot10$ for some integer~$n\ge0$ (otherwise~$p$ would be even). We consider the partial Hasse invariant~$h_{p,2}$ of weight~$(p,-1)$, which is the one associated with the non-trivial component of the modular embedding, and whose treatment is more delicate. One has
\[
\Bigl\{\frac{-k_1+3k_2}{10}\Bigr\}\=\Bigl\{\frac{-p-3}{10}\Bigr\}=\Bigl\{\frac{-6-10\cdot n}{10}\Bigr\}\=\frac{2}{5}\,.
\]
This means that for~$p\equiv 3\mod 5$, the elliptic point of order five always reduce to a non-ordinary point modulo~$p$, and that the partial Hasse invariant has multiplicity~$2/5$ at this point. The degree of the partial Hasse polynomial for such primes is then 
\[
\mathrm{deg}(\mathrm{ph}_2(t))\=\mathrm{deg}(\mathrm{ph}^*_2(t))\+\epsilon\+1\,,
\]
where~$\epsilon=1$ if the elliptic point of order two is non-ordinary and zero otherwise, and~$1$ counts the elliptic point of order five. By using Corollary~\ref{cor:degree} and~\eqref{eq:dim5} and the computation above one finds
\[
\mathrm{deg}(\mathrm{ph}_2(t))\=\frac{3p-1}{20}\-\frac{\epsilon}{2}\-\frac{2}{5}\+\epsilon\+1\=\frac{3p-1}{20}+\frac{\epsilon}{2}+\frac{3}{5}\,,
\]
where~$\epsilon=1$ if~$p\equiv 1\mod 4$ and zero otherwise. 
Reasoning as above in the remaining cases one proves the following result.
\begin{proposition}
\label{prop:W5}
The partial Hasse polynomials of the triangle curve~$W_5=\Po/\Delta(2,5,\infty)\subset\Po^2/(\mathcal{O}_5^\vee\oplus\mathcal{O}_5)$ have the following degrees:
\begin{enumerate}
\item if~$p$ is inert in~$\Q(\sqrt{5})$,
\[
\deg{\mathrm{ph}_1(t)}=\frac{p-3}{20}+\frac{\epsilon}{2}+\frac{5-\delta_1}{5}\,,\quad \deg{\mathrm{ph}_2(t)}=\frac{3p-1}{20}+\frac{\epsilon}{2}+\frac{5-\delta_2}{5}\,,
\]
where~$\epsilon=1$ if~$p\equiv1\mod 4$ and~$\epsilon=0$ otherwise, and~$(\delta_1,\delta_2)=(1,5)$ if~$p\equiv2\mod5$ and $(\delta_1,\delta_2)=(5,2)$ if~$p\equiv3\mod5$ ($p$ is inert in~$\Q(\sqrt{5})$ if and only if~$p\equiv2,3\mod5)$. 

\item If~$p$ is split in~$\Q(\sqrt{5})$, then~$p\equiv 1, 4\mod5$, and the same arguments show that 
\[
\deg{\mathrm{ph}_1(t)}=\frac{3(p-1)}{20}+\frac{\epsilon}{2}+\frac{5-\delta_1}{5}\,,\quad \deg{\mathrm{ph}_2(t)}=\frac{p-1}{20}+\frac{\epsilon}{2}+\frac{5-\delta_2}{5}\,,
\]
where~$\epsilon=0$ 
If~$p\equiv1\mod4$ and~$\epsilon=1$ otherwise, and~$(\delta_1,\delta_2)=(5,5)$ if~$p\equiv1\mod5$ and $(\delta_1,\delta_2)=(1,2)$ if~$p\equiv4\mod5$. 
\end{enumerate}
\end{proposition}
Inequalities for the cardinality of the non-ordinary, supersingular, and superspecial locus in~$\Po/\Delta(2,5,\infty)$ can be obtained as in the previous examples.

\section{Construction of partial Hasse invariants from explicit models}
\label{sec:expl}
The goal of this section is to construct partial Hasse invariants as polynomial in modular forms, in the case where an  integral model of the curve~$Y\hookrightarrow M_F$ is given and can be parametrized by modular forms. Basically, we write the entries of the Hasse-Witt matrix in terms of modular forms. We illustrate the method in the non-arithmetic curve~$Y=\Po/\Delta(2,5,\infty)$. 

The Fuchsian group~$\Delta(2,5,\infty)$ is related to the uniformization of the curve~$W_5\subset\mathcal{M}_2$ (where~$\mathcal{M}_2$ is the moduli space of smooth genus 2 curves) parametrizing the family of genus two curves~$C_\eta$ given by
\begin{equation}
\label{eq:ext}
C_\eta\;:\;\begin{cases}
y^2&=x^5-5x^3+5x-2\eta,\quad \eta\neq\infty\\
y^2&=x^5-1,\quad \eta=\infty\,.
\end{cases}
\end{equation}
The Jacobian of~$C_\eta$ has real multiplication by~$\Q(\sqrt{5})$. More precisely, $W_5$ is the algebraically primitive Teichmüller curve in~$\mathcal{M}_2$ of discriminant~$5$, i.e., the locus of genus two curves whose Jacobian has real multiplication by~$\Q(\sqrt{5})$ and a holomorphic one form with a double zero in a Weierstrass point (this description plays no role in the following). In particular, $W_5$ can be embedded in the Hilbert modular surface~$\Po^2/\SL(\mathcal{O}_5^\vee\oplus\mathcal{O}_5)$. 

The algebraic structure of the space of twisted modular forms on~$\Delta(2,5,\infty)$ has been determined in~\cite{BN}. It is the ring of Laurent polynomials~$M_*(\Delta(2,5,\infty))[B^{\pm1}]$ over the ring of ordinary modular forms~$M_*(\Delta(2,5,\infty))$, where~$B$ is a twisted modular form of weight~$(-\tfrac{1}{3},1)$ which has no zeros on~$\overline{\Po}$.
Moreover~$M_*(\Delta(2,5,\infty))=\C[Q,R]$, where~$Q$ and~$R$ are modular forms (with respect to a non-trivial multiplier system) of weight~$\tfrac{4}{3}$ and~$\tfrac{10}{3}$ respectively. It follows that~$M_{(*,*)}(\Delta(2,5,\infty),\varphi)=\C[Q,R,B^{\pm1}]$. 

It holds~$\mathrm{div}(Q)=\tfrac{1}{5}\cdot e^{\pi i/5}$ and~$\mathrm{div}(R)=\tfrac{1}{2}\cdot i$. and then the Hauptmodul~$t=\frac{Q^5-R^2}{Q^5}$, satisfies~$t(i\infty)=0$, and~$t(i)=1$, and~$t(e^{\pi i/5})=\infty$.

\begin{theorem}
\label{thm:phi3}
Let~$Q(\tau)$ and~$R(\tau)$ be the generators of the ring of modular forms of~$\Delta(2,5,\infty)$, and let~$B(\tau)$ the twisted modular form of weight~$(-1/3,1)$ introduced above. 
Consider the power series expansion
\begin{equation}
\label{eq:gs2}
\bigl(1-5Q(\tau)x^4+5Q(\tau)^2x^8-2R(\tau)x^{10}\bigr)^{-1/2}=\sum_{n\ge0}{d_n(Q(\tau),R(\tau))x^n}\,,
\end{equation}
whose~$n$-th coefficient is a modular form of weight~$\tfrac{n}{3}$. For every prime~$p$ of good reduction, the twisted modular forms
\[
h_{p,1}(\tau)=B^p(\tau)\cdot d_{p-3}(Q(\tau),R(\tau))\quad\text{and}\quad h_{p,2}(\tau)=B(\tau)^{-1}\cdot d_{3p-1}(Q(\tau),R(\tau))
\]
if~$p$ is inert in~$\Q(\sqrt{5})$, or, if~$p$ is split in~$\Q(\sqrt{5})$,
\[
h_{p,1}(\tau)=d_{3p-3}(Q(\tau),R(\tau))\quad\text{and}\quad h_{p,2}(\tau)=B(\tau)^{p-1}\cdot d_{p-1}(Q(\tau),R(\tau))
\]
are lifts of the partial Hasse invariants for the curve~$\Po/\Delta(2,5,\infty)\hookrightarrow\mathcal{M}_{\mathbb{Q}(\sqrt{5})}(\C)$. 
\end{theorem}
\begin{proof}
In the proof we will use letters~$Q,R$ to denote parameters, and~$Q(\tau)$ and~$R(\tau)$ to denote modular forms as in the statement. 
We first rewrite the equation of~$C_\eta$ in~\eqref{eq:ext} as
\begin{equation}
\label{eq:exQR}
y^2\=x^5\-5Qx^3\+5Q^2x\-2R\,
\end{equation}
via the transformation~$\eta=RQ^{-5/2}$ and~$(x,y)\mapsto(xQ^{-1/2},yQ^{5/2})$. This is motivated by the following considerations: the parameter~$\eta$ can be regarded as a modular function on an index-two subgroup of~$\Delta(2,5,\infty)$, and it holds~$\eta(\tau)^2=1-t(\tau)=R(\tau)^2Q(\tau)^{-5}$. The appearance of~$\eta^2$ is natural because~$C_{\eta}$ and~$C_{-\eta}$ are isomorphic, while the shift of~$1$ depends on our normalization of~$t$. 

Fix~$p>5$ a prime, and consider~$Q$ and~$R$ in~\eqref{eq:exQR} as parameters varying over an extension of~$\F_p$. The power series expansion
\begin{equation}
\label{eq:gs1}
(x^5\-5Qx^3\+5Q^2x\-2R)^{(p-1)/2}=\sum_{l=0}^N{c_l(Q,R)x^l}\;\in\;\overline{\F}_p[x]\,.
\end{equation}
has coefficients~$c_l(Q,R)$ that are are polynomials in~$Q,R$. 
It is known since the work of Manin~\cite{Manin} that the Hasse-Witt matrix of the family~\eqref{eq:exQR} is given by
\begin{equation}
\label{eq:HW}
M(Q,R)\:=\begin{pmatrix}
c_{p-1}(Q,R)& c_{p-2}(Q,R)\\
c_{2p-1}(Q,R)& c_{2p-2}(Q,R)
\end{pmatrix}\,.
\end{equation}
Because of the action of real multiplication, it follows that~$c_{p-1}(Q,R)=c_{2(p-1)}(Q,R)=0$ if~$p$ is inert in~$\Q(\sqrt{5})$, and that~$c_{p-2}(Q,R)=c_{2p-1}(Q,R)=0$ if~$p$ is split in~$\Q(\sqrt{5})$ (see~\cite{B} for a general result). It follows that for~$j=1,2$ the~$j$-th  component of the non-ordinary locus of the modulo~$p$ reduction of the family~\eqref{eq:ext} is the zero locus of~$c_{j(p-1)}(Q,R)$ if~$p$ is split, and is the zero locus of~$c_{(j+1)p-j}(Q,R)$, where~$j+1=1$ if~$j=2$, if~$p$ is inert. 

The polynomial~$(1-5Qx^4+5Qx^8-2Rx^{10})$, which for~$Q=Q(\tau)$ and~$R=R(\tau)$ is the one in~\eqref{eq:gs2}, is obtained from the one in~\eqref{eq:gs1} by substituting~$x\to x^{-2}$ and multiplying by~$x^{10\frac{(p-1)}{2}}$. This implies in particular the following relations between expansion coefficients:
\begin{align*}
&c_{p-1}(Q,R)\equiv d_{3p-3}(Q,R) \mod p\,,\quad &c_{p-2}(Q,R)\equiv d_{3p-1}(Q,R) \mod p\,,\\
&c_{2p-1}(Q,R)\equiv d_{p-3}(Q,R) \mod p\,,\quad &c_{2p-2}(Q,R)\equiv d_{p-1}(Q,R) \mod p\,.
\end{align*}
It finally follows that, in the inert case, the zero locus of the modular forms~$d_{p-3}(Q(\tau),R(\tau))$ and of~$d_{3p-1}(Q(\tau),R(\tau))$ modulo~$p$ is the zero locus of the partial Hasse invariants~$h_{p,1}$ and~$h_{p,2}$ on~$\Po/\Delta(2,5,\infty)$ respectively. They are not lifts of partial Hasse invariant though, since their weights are~$(\tfrac{p}{3}-1,0)$ and~$(p-\tfrac{1}{3},0)$. Nevertheless, multiplication by suitable powers of~$B(\tau)$, which does not modify the zero locus, yields twisted modular forms of correct weight. In particular~$B(\tau)^pd_{p-3}(Q(\tau),R(\tau))$ and~$B(\tau)^{-1}d_{3p-1}(Q(\tau),R(\tau))$ are lifts of partial Hasse invariants of weight~$(-1,p)$ and~$(p,-1)$ respectively. 
Similarly, $d_{3p-3}(Q(\tau),R(\tau))$ and~$B(\tau)^{p-1}d_{p-1}(Q(\tau),R(\tau))$ are lifts of the partial Hasse invariants of weight~$(p-1,0)$ and~$(0,p-1)$ respectively in the split case. 
\end{proof}

We conclude with an example of the~$t$-expansion at~$i\infty$ of the partial Hasse invariants of Theorem~\ref{thm:phi3}, starting with the~$t$-expansion of the building blocks~$Q(\tau)$ and~$R(\tau)$ and~$B(\tau)$.
These are related to hypergeometric functions
\[
{}_2F_1(a,b;1;t)\=\sum_{n\ge0}\frac{(a)_n(b)_n}{(n!)^2}t^n\,, 
\]
where~$(x)_n=x(x+1)\cdots(x+n-1)$ is the Pochhammer symbol, which are the integral solutions of the Picard-Fuchs differential equations associated to the curve~$\Po/\Delta(2,5,\infty)$. More precisely, it follows from Corollary~3.3 of~\cite{BN} (with~$n=2, m=3$, and~$j=1$ and~$k_1=r_1=1)$ that
\begin{equation}
\label{eq:exQ}
Q(\tau)^3\={}_2F_1\biggl(\frac{7}{20},\frac{3}{20};1;t\biggr)^{4}\=1 \+ \frac{21}{100}t \+ \frac{15687}{160000}t^2 \+ \frac{236691}{4000000}t^3\+\cdots\,.
\end{equation}
Next, the relation~$t=(Q^5-R^2)/Q^5$ implies that~$R(\tau)^2\=(1-t)Q(\tau)^5$ and that the~$t$-expansion of~$R$ can be computed from the~$t$-expansion of~$Q$. 

Finally, let~$\varphi=(1,\varphi_2(\tau))$ denote the modular embedding for the group~$\Delta(2,5,\infty)$. 
Corollary~3.4 of~\cite{BN}, with~$n=2,m=3$ and~$j=2$ and~$k_2=2, r_1=1$, implies that
\begin{equation}
\label{eq:exP}
\varphi_2'(\tau)\=\frac{{}_2F_1\Bigl(\frac{7}{20},\frac{3}{20};1;t\Bigr)^2}{{}_2F_1\Bigl(\frac{9}{20},\frac{1}{20};1;t\Bigr)^2}\=1 + \frac{3}{50}t + \frac{927}{40000}t^2 + \frac{20757}{1600000}t^3\cdots\,.
\end{equation}
The function~$\varphi_2'(\tau)$ has a zero of order~$\tfrac{1}{5}$ at the class of the elliptic point~$e^\frac{\pi i}{5}\in\Po$, precisely like~$Q$. It follows that the twisted modular form~$B=B_2$ defined in Proposition~\ref{prop:Bj} satisfies~$B^2=\frac{Q}{\varphi_2'}$, and its~$t$-expansion 
\begin{equation}
\label{eq:exB}
B(\tau)\=1 + \frac{1}{100}t + \frac{641}{160000}t^2 + \frac{54607}{24000000}t^3+\cdots\,.
\end{equation}
can be computed from~\eqref{eq:exQ} and~\eqref{eq:exP}. 

Let~$p=11$; it is split in~$\Q(\sqrt{5})$. With the notation of Theorem~\ref{thm:phi3} the lifts of the partial Hasse invariants for~$\Delta(2,5,\infty)$ are
\begin{equation}
\label{eq:ex11}
\begin{aligned}
h_{11,1}(\tau)&\=d_{30}(Q(\tau),R(\tau))=\frac{800625}{256}Q(\tau)^5R(\tau)+\frac{5}{2}R(\tau)^3\,,\\
\quad h_{11,2}(\tau)&\=B(\tau)^{10}d_{10}(Q(\tau),R(\tau))\=B(\tau)^{10}R(\tau)\,.
\end{aligned}
\end{equation}
We can easily compute the~$t$-expansion of~$h_{11,1}(\tau)$ is closed form if we consider its square. By exploiting the relations~\eqref{eq:exQ} and~$R^2=(1-t)Q^5$ one finds
\[
\begin{aligned}
h_{11,1}(\tau)^2&\=\Bigl(\frac{800625}{256}\Bigr)^2Q(\tau)^{10}R(\tau)^2+\frac{5\cdot800625}{256}Q(\tau)^5R(\tau)^4+\frac{25}{4}R(\tau)^6\\
&\=Q(\tau)^{15}\biggl[\Bigl(\frac{800625}{256}\Bigr)^2\frac{R(\tau)^2}{Q(\tau)^5}+\frac{5\cdot800625}{256}\frac{R(\tau)^4}{Q(\tau)^{10}}+\frac{25}{4}\frac{R(\tau)^6}{Q(\tau)^{15}}\biggr]\\
&\={}_2F_1\biggl(\frac{7}{20},\frac{3}{20};1;t\biggr)^{20}\biggl[\Bigl(\frac{800625}{256}\Bigr)^2(1-t)+\frac{5\cdot800625}{256}(1-t)^2+\frac{25}{4}(1-t)^3\biggr]\\
&\={}_2F_1\biggl(\frac{7}{20},\frac{3}{20};1;t\biggr)^{20}\biggl[\frac{642025600225}{65536}-\frac{643051219425}{65536}t+\frac{4007925}{256}t^2-\frac{25}{4}t^3\biggr]\\
&\= \frac{642025600225}{65536}\+ \frac{124302643245}{262144}t\- \frac{97678687599009}{83886080}t^2 \+\cdots\\
&\equiv {}_2F_1\biggl(\frac{7}{20},\frac{3}{20};1;t\biggr)^{20}(1+5t^11+3t^2+2t^3)\equiv 1\mod 11\,.
\end{aligned}
\]
For the lift of the second partial Hasse invariant, by using the relations~$R^2=(1-t)Q^5$ and~$B^2=Q(\varphi_2')^{-1}$ and the hypergeometric relations~\eqref{eq:exQ} and~\eqref{eq:exP}, one proves that
\[
\begin{aligned}
h_{11,2}(\tau)^2&\=R(\tau)^2B(\tau)^{20}\=(1-t)Q(\tau)^5\frac{Q(\tau)^{10}}{(\varphi_2')^{10}}\\
&\=(1-t){}_2F_1\biggl(\frac{7}{20},\frac{3}{20};1;t\biggr)^{20}\cdot\frac{{}_2F_1\biggl(\frac{9}{20},\frac{1}{20};1;t\biggr)^{20}}{{}_2F_1\biggl(\frac{7}{20},\frac{3}{20};1;t\biggr)^{20}}\=(1-t)\cdot{}_2F_1\biggl(\frac{9}{20},\frac{1}{20};1;t\biggr)^{20}\\
&\=1 \-\frac{11}{20}t \-\frac{ 5841}{32000}t^2 \-\frac{ 68541}{800000}t^3\-\frac{3937395297}{81920000000}t^4\+\cdots\equiv 1\mod 11\,.
\end{aligned}
\]

For the inert prime~$p=13$ similar computations prove that
\[
\begin{aligned}
h_{13,1}(\tau)^2&\=\bigl(B(\tau)^{13}d_{10}(Q(\tau),R(\tau))\bigr)^2\=B(\tau)^{26}R(\tau)^2\=(1-t)\cdot\frac{{}_2F_1\Bigl(\frac{9}{20},\frac{1}{20};1;t\Bigr)^{26}}{{}_2F_1\Bigl(\frac{7}{20},\frac{3}{20};1;t\Bigr)^2}\,,\\
h_{13,2}(\tau)^2&=\bigl(B(\tau)^{-1}d_{38}(Q(\tau),R(\tau))\bigr)^2
=p(t)\cdot \frac{{}_2F_1\Bigl(\frac{7}{20},\frac{3}{20};1;t\Bigr)^{26}}{{}_2F_1\Bigl(\frac{9}{20},\frac{1}{20};1;t\Bigr)^2}\,,
\end{aligned}
\]
where
\[
p(t)=\Bigl(\frac{100321875}{2048}\Bigr)^2(1-t)+\frac{719809453125}{32768}(1-t)^2+\Bigl(\frac{7175}{16}\Bigr)^2(1-t)^3\,.
\]

The appearance of only one hypergeometric series for each partial Hasse invariant in the split~$p=11$ case, and the appearance of two hypergeometric series in the inert~$p=13$ case, can be explained in terms of weights. In the split case the partial Hasse invariant has weight concentrated in one component, i.e., $(p-1,0)$ and~$(0,p-1)$, while in the inert case it is mixed~$(-1,p)$ and~$(p,-1)$. The hypergometric series~${}_2F_1\bigl(\frac{7}{20},\frac{3}{20};1;t\bigr)$ , via the modularity of the Picard-Fuchs differential equation (see Section~1 of~\cite{B}), is related to modular forms of weight~$(*,0)$; the hypergeometric series~${}_2F_1\bigl(\frac{9}{20},\frac{1}{20};1;t\bigr)$ to twisted modular forms of weight~$(0,*)$. This phenomenon is already visible in the formula~\eqref{eq:exP} for~$\phi_2'$, which is a twisted modular form of weight~$(2,-2)$.

\bibliography{Hasse}{}
\bibliographystyle{plain}
\end{document}